%% file: pap_phi_arxiv.tex
\documentclass{siamart190516}
\usepackage{amsfonts,amsmath,amssymb,hyperref,xcolor}

\usepackage{changes}

\title{A residual concept for Krylov subspace evaluation of
the $\varphi$ matrix function}

\author{%
M.A. Botchev%
\thanks{Corresponding author. Keldysh Institute of Applied Mathematics,
Russian Academy of Sciences, Miusskaya~Sq.~4, Moscow 125047,
Russia, \email{botchev@ya.ru}.}
\and
L.A. Knizhnerman%
\thanks{Mathematical Modelling Department, Central Geophysical Expedition,
Narodnogo Opolcheniya St., 38, Bldg.~3,
Moscow 123298, Russia, \email{lknizhnerman@gmail.com}.}
\and
E.E. Tyrtyshnikov%
\thanks{Marchuk Institute of Numerical Mathematics, Russian Academy of Sciences,
Gubkin St.~8, Moscow 119333, Russia, \email{eugene.tyrtyshnikov@gmail.com}.
Work of this author is supported by the Russian Science Foundation grant
No.~19-11-00338.}
}

\input{L_newcommands.tex}

\newcommand{\Cc}{\mathbb{C}}
\newcommand{\conv}{\mathtt{convergence}}
\newcommand{\epsi}{\epsilon}
\newcommand{\geqs}{\geqslant}
\newcommand{\Hh}{\underline{H}}
\newcommand{\Kk}{\mathcal{K}}
\newcommand{\leqs}{\leqslant}
\newcommand{\Pe}{\mathit{P\!e}}
\newcommand{\phip}{\texttt{phip}}
\newcommand{\phipm}{\texttt{phipm}}
\newcommand{\phiRT}{\texttt{phiRT}}
\newcommand{\phiv}{\texttt{phiv}}
\renewcommand{\Re}{\mathrm{Re}}
\newcommand{\Rr}{\mathbb{R}}
\newcommand{\Rrnn}{\mathbb{R}^{n\times n}}
\newcommand{\tol}{\mathtt{tol}}

\begin{document}
\maketitle

\begin{abstract}
  An efficient Krylov subspace algorithm for computing actions
  of the $\varphi$ matrix function for large matrices is proposed.
  This matrix function is widely used in exponential time integration,
  Markov chains and network analysis and many other applications.
  Our algorithm is based on a reliable residual based stopping criterion and
  a new efficient restarting procedure.
  For matrices with numerical range in the stable complex half plane,
  we analyze residual convergence and prove
  that the restarted method is guaranteed to converge for
  any Krylov subspace dimension.
  Numerical tests demonstrate efficiency of our approach for solving
  large scale evolution problems resulting from discretized in space
  time-dependent PDEs, in particular, diffusion and convection--diffusion
  problems.
\end{abstract}

\begin{keywords}
  phi matrix function, Krylov subspace methods, exponential time integration,
  matrix exponential, restarting, exponential residual
\end{keywords}

\begin{AMS}
  65F60;  
  65M20;  
  65L05   
\end{AMS}

\section{Introduction}
Exponential time integration is an actively developing research
field~\cite{HochbruckOstermann2010}, with numerous successful applications in
challenging large scale computations, such as
elastic wave equations~\cite{Hochbruck_Pazur_ea2015},
wind farm simulations~\cite{Fu_ea2019},
power delivery network analysis~\cite{Wang_ea2019},
large circuit analysis~\cite{Zhuang_ea2016}
vector finite element discretizations of Maxwell's equations~\cite{Gautschi2006}
and
photonic crystal modeling~\cite{Botchev2016,BotchevHanseUppu2018}.

In this paper we propose a Krylov subspace method for solving initial-value
problem (IVP)
\begin{equation}
\label{ivp}
y'(t) = -Ay(t) + g,\quad y(0) = v,\quad t\geqs 0,
\end{equation}
where the matrix $A\in\Rr^{n\times n}$ and vectors $g,v\in\Rr^n$
($g\ne 0$) are given
and the problem size $n$ is supposed to be large, i.e., $n\gg 1$.
Using the variation of constants formula (see,
e.g.,~\cite[Chapter~2.3]{HundsdorferVerwer:book})
$$
y(t) = \exp(-tA)v + \left[\int_0^t \exp(-(t-s)A)\,d s \right]g,
\quad t\geqs 0,
$$
we can write
$y(t)$ in the form
\begin{equation}
\label{phim}
y(t) = v + t\varphi(-tA)(g-Av),\quad t\geqs 0.
\end{equation}
Here $\varphi(-tA)(g-Av)$ is a matrix-vector product with the
matrix function $\varphi(-tA)$ and $\varphi$ is defined as
\begin{equation}
\label{phi}
\varphi(z)\equiv \frac{e^z-1}{z}, \quad z\in\Cc, 
\end{equation}
where we set $\varphi(0)=1$ so that $\varphi$ is an entire function.
Thus, the matrix function $\varphi$ provides the
exact solution~\eqref{phim} to system~\eqref{ivp}.
In other words, solving IVP~\eqref{ivp} is equivalent to
evaluating the matrix function $\varphi$ times a vector in~\eqref{phim}.
The method for solving~\eqref{ivp} we propose in this paper
is based on an efficient Krylov subspace evaluation of the matrix function
$\varphi(-tA)$ in~\eqref{phim}.

For homogeneous ODEs (ordinary differential equations), i.e., for $g=0$,
the exact solution of~\eqref{ivp} takes
the well-known form $y(t)=\exp(-tA)v$.
Therefore,
computing matrix-vector products with the matrix function $\varphi$
and the matrix exponential is an important task in exponential
time integration~\cite{HochbruckOstermann2010},
occurring independently as well as within higher-order ODE Krylov subspace
methods~\cite{Hochbruck_ea2009,Gaudreault_ea2018_kiops} or
splitting schemes~\cite{HansenOstermann2016,BotchevFaragoHorvath2009}.
For large problems a direct evaluation of the matrix
function based on dense linear algebra
techniques~\cite{19ways,Higham_bookFM} is out of question.
Therefore, the matrix-vector products with the matrix
function are approximately computed, rather than the matrix function
itself.
Krylov subspace methods have been a successful tool for doing
this since the end of the eighties;
we list
chronologically~\cite{ParkLight86,Henk:f(A),DruskinKnizh89,Knizh91,Saad92,%
DruskinKnizh95,HochLub97}.
This progress in numerical linear algebra has triggered advances
in numerical time integration
methods (see, e.g.,~\cite{GallSaad92,HochLubSel97,CelledoniMoret97})
and a revival of the exponential time integrators~\cite{HochbruckOstermann2010},
which have been developed since
the sixties~\cite{Certaine1960,Legras1966,Lawson1967}.

In this work we are interested in solving large scale problems~\eqref{ivp}
which result, for instance, within a method of lines,
when spatially discretized PDEs (partial differential equations)
have to be efficiently solved in time.
This means that $A$ does not have to be symmetric and is often far from a normal matrix.
Furthermore, fine and nonuniform grids as well as realistic coefficient values
in the PDE may lead to a stiffness in the problem, i.e., the eigenvalues
of $A$ may significantly vary in magnitude~\cite{Fedorenko94stiff}.
Taking into account unavoidable space discretization errors,
error produced by a time integrator does not have to be extremely small
in these problems. Usually absolute time errors in the range $[10^{-7},10^{-3}]$
suffice and this is the range we aim here.
As we see in numerical tests, delivering accuracies in this range
can be a challenge for existing exponential solvers.
The proposed method appears to carry out this task successfully.

Krylov subspace exponential integrators are good candidates
for solving large scale problems with sparse nonnormal matrices $A$,
such as space-discretized PDEs.
Indeed, these solvers have sufficient stability properties
and often lead to an efficient compromise between standard implicit time
integrators (which may occur too expensive in large scale applications)
and explicit stabilized solvers.
Moreover, as Krylov subspace methods 
can significantly profit from eigenvalue clustering~\cite{Henk:book},
their application for solving stiff problems can be
successful~\cite{GearSaad,Chan86:ODE,VODPK,MRAIpap}
because stiffness typically
means a large discrepancy in
eigenvalues.
Other approaches, not employing Krylov subspaces, for evaluating
matrix exponential and related matrix functions of large sparse matrices
include the Chebyshev polynomials, scaling and squaring with Taylor or Pad\'e
approximations and contour integrals~\cite{TalEzer89,DeRaedt03,SchmelzerTrefethen07,CaliariOstermann09,AlmohyHigham2011}.

Our approach for computing $\varphi$ matrix-vector products
is essentially based on two ingredients presented in this paper:
(i)~a residual notion
for $\varphi$ matrix functions and (ii)~a restarting,
i.e., an algorithm aiming to preserve
convergence while keeping the Krylov dimension $k$ bounded,
$k\leqs k_{\max}$~\cite{EiermannErnst06}.
The residual notion leads to a reliable stopping criterion
and facilitates the proposed restarting. 
Restarting is often indispensable for solving large scale problems
because the size $n$ of the matrix $A$ tends to be extremely large
in real life applications.  Thus, storing $k_{\max}$ Krylov basis vectors
of size~$n$ can be prohibitively expensive for $n\gg 1$
even if $k_{\max}$ is moderately large.
Therefore, in the tests we restrict $k_{\max}$ by $30$.
Numerous restarting algorithms for Krylov subspace
evaluations of the matrix exponential have been proposed,
for instance~\cite{Afanasjew_ea08,CelledoniMoret97,Eiermann_ea2011,PhD_Niehoff,BGH13}.
Restarting for general matrix functions
include~\cite{EiermannErnstGuttel11,GuettelFrommerSchweitzer2014}.
Two restarting techniques have been designed specifically for the $\varphi$
matrix functions, namely~\cite{EXPOKIT} and~\cite{NiesenWright12}.
They both are based on error estimates and choose a time interval length $t$
such that solving~\eqref{ivp} by $k_{\max}$ Krylov
steps yields a sufficiently small error.  In the numerical experiments
presented below we show that our method performs favorably compared
to the codes based on these techniques.

Furthermore, we provide convergence analysis of the method,
obtaining upper bounds for the residual norm with no asymptotic
estimates involved, and show that the
restarted algorithm is guaranteed to converge for any restart
length, i.e., for any Krylov dimension~$k_{\max}$.

The structure of this paper is as follows.
In Section~\ref{s:setting} we
describe details of the Krylov subspace approximation and
introduce a residual concept for the $\varphi$ matrix function.
In Section~\ref{s:rst} some observations on the
residual behavior are given and the restarting algorithm,
based on the observed behavior, is formulated.
Section~\ref{s:Faber} is devoted to residual convergence
analysis for both symmetric and nonsymmetric matrices.
Based on this analysis, in Section~\ref{s:conv} we show
convergence of the restarted algorithm.
Section~\ref{s:num} presents numerical experiments.
Finally, some conclusions are drawn in Section~\ref{s:concl}.

\section{Krylov subspace evaluations of the $\varphi$ matrix function}
\label{s:setting}
\subsection{Residual notion for the $\varphi$ matrix function}
\label{s:Phi_res}
Throughout the paper we assume that
\begin{equation}
\label{posdef}
\Re (z^*Az)\geqs 0,\quad \forall z\in\Cc^n,  
\end{equation}
where the function $\Re(\cdot)$ is the real part of its argument,
and, unless reported otherwise,
$\|\cdot\|$ denotes the vector or matrix 2-norm.

To keep our presentation simpler, we first assume that
the initial value vector in~\eqref{ivp} is zero: $v=0$.
In this case relation~\eqref{phim} takes a form
\begin{equation}
  \tag{\ref{phim}$'$}
y(t) = t\varphi(-tA)g,\quad t\geqs 0,  
\end{equation}
and, thus, to solve~\eqref{ivp} we have to evaluate the action
of the matrix function $\varphi(-tA)$ on the vector $g$.
Let upper Hessenberg matrix $\Hh_k\in\Rr^{(k+1)\times k}$ and
matrix
$$
V_{k+1}=
\begin{bmatrix}
v_1 & v_2 & \dots & v_k & v_{k+1}  
\end{bmatrix}
\in\Rr^{n\times(k+1)}, \qquad v_1 = g/\|g\|,
$$
with orthonormal columns $v_j$, be the matrices
which result
after carrying out $k$ steps of the Arnoldi process.
These matrices 
satisfy the Arnoldi decomposition~\cite{GolVanL,Parlett:book,Henk:book,SaadBook2003}
\begin{equation}
\label{Arn}
AV_k=V_{k+1}\Hh_k = V_kH_k + h_{k+1,k}v_{k+1}e_k^T.
\end{equation}
Here $H_k=V_k^TAV_k\in\Rr^{k\times k}$ consists of the first $k$ rows of $\Hh_k$
and the columns $v_1$, \dots, $v_k$ span the Krylov subspace
$$
\Kk_k(A,g) = \mathrm{span} (g, Ag, \dots, A^{k-1}g).
$$
A standard way to compute a Krylov subspace approximation $y_k(t)$ to $y(t)$
in~(\ref{phim}$'$)
is to set~\cite[Chapter~11]{Henk:book}
\begin{equation}
\label{yk}
y_k(t):= V_k t \varphi(-tH_k) (\beta e_1), 
\end{equation}
with $\beta=\|g\|$ and $e_1=[ 1, 0, \dots, 0]^T\in\Rr^k$.
This expression for $y_k(t)$ can be derived by 
projecting the IVP~\eqref{ivp} onto the Krylov
subspace $\Kk_k(A,g)=\mathrm{colspan} (V_k)$
in the Galerkin sense.
Indeed, let us define
a residual $r_k(t)$ of $y_k(t)$ with respect to~\eqref{ivp}, i.e.,
let 
\begin{equation}
\label{res}
r_k(t)\equiv -Ay_k(t)+g - y_k'(t).
\end{equation}
Now consider
$$
y_k(t)\approx y(t), \qquad y_k(t)=V_ku(t)
$$
such that its residual 
is kept orthogonal to the Krylov subspace: $V_k^Tr_k(t)=0$.
Since $g=\beta V_ke_1$, this
orthogonality condition leads to a projected IVP in
$u(t): \Rr\rightarrow\Rr^k$
\begin{equation}
\label{ivp1}
u'(t) = -H_k u(t) + \beta e_1,\qquad u(0) = 0,
\end{equation}
where $u(0)=0$ follows from the assumption $y(0)=0$.
With $u(t)=t \varphi(-tH_k) (\beta e_1)$,
we see that $V_ku(t)$ yields \eqref{yk}.

Now it is easy to observe that
the residual~\eqref{res} is readily
computable in the course of the Arnoldi process:
\begin{equation}
\label{rk}
\begin{aligned}
  r_k(t) &\stackrel{\eqref{yk}}  {=} -AV_ku(t)+g - V_ku'(t)
          \stackrel{\eqref{ivp1}}{=} -AV_ku(t)+g- V_k(-H_k y(t) + \beta e_1)\\
          &= -AV_ku(t) + g + V_kH_ku(t) - g
          \stackrel{\eqref{Arn}}{=} -h_{k+1,k}v_{k+1}e_k^Tu(t),
\end{aligned}
\end{equation}
where the bracketed numbers above the equality signs
indicate the relations used in the derivation.

In case the initial value vector $v$ in~\eqref{ivp} is not zero,
we introduce a function $\tilde{y}(t)=y(t)-v$ which satisfies an IVP
$$
\tilde{y}'(t)=-A\tilde{y}(t) + \tilde{g},
\quad \tilde{y}(0)=0,\quad t\geqs 0,
$$
with $\tilde{g}=g-Av$. Using~\eqref{phim} for this transformed problem
we see
that $\tilde{y}(t) = t\varphi(-tA)\tilde{g}$ and, thus, a Krylov
subspace approximation $\tilde{y}_k(t)$ to $\tilde{y}(t)$,
as well as its residual $\tilde{r}_k(t)$ (cf.~\eqref{res}),
can be computed as discussed above.  Our Krylov subspace
approximation to $y(t)$ is then $y_k(t)=\tilde{y}_k(t)+v$.
For this approximation we introduce the residual $r_k(t)$ in the same way as
for the case $v=0$: see~\eqref{res}.  Then we have
$$
r_k(t) = -Ay_k(t)+g-y_k'(t) = -A\tilde{y}_k(t)-Av+g-y_k'(t)
= -A\tilde{y}_k(t)+\tilde{g}-\tilde{y}_k'(t)=\tilde{r}_k(t),
$$
i.e., the residual $r_k(t)$ coincides with the
residual $\tilde{r}_k(t)$ for the transformed solution
$\tilde{y}_k(t)$ and, hence, is also easily computable.

Note that the error $\epsi_k(t)\equiv y(t)-y_k(t)$ of
the Krylov approximation $y_k(t)$ satisfies an IVP
\begin{equation}
  \label{err}
  \epsi_k'(t) = -A \epsi_k(t) + r_k(t), \quad \epsi_k(0)=0.
\end{equation}
This shows that the introduced residual can be seen as the backward
error for the approximate solution $y_k(t)$.  This is similar to the well known
property of the linear system residual:
if $\tilde{x}$ approximates exact solution $x$ of a linear
system $Ax=b$ then for its residual $\tilde{r}=b-A\tilde{x}$ 
we have $A(x-\tilde{x})=\tilde{r}$.
Using~\eqref{posdef},\eqref{err} we can
bound the error norm as \cite[Lemma~4.1]{BGH13}
\begin{equation}
  \label{err_bnd}
  \|\epsi_k(t)\|\leqs t\varphi(-t\omega)\max_{s\in[0,t]}\|r_k(s)\|,
\end{equation}
where $\omega\geqs 0$ is a constant introduced below in~\eqref{omega}.
Hence, any residual estimates, including obtained below in Section~\ref{s:Faber}
can be regarded as error estimates as well. 
This backward error property is also shared by the matrix exponential
residual, see e.g.~\cite{BGH13}.

\subsection{$\varphi$ via the matrix exponential, two approaches}
If $A$ is nonsingular, a naive approach to compute $y:=\varphi(-A)g$
for a given $g$ would be to solve $v$ from $Av=-g$ and evaluate
$y:=\exp(-A)v - v$.  This means that we replace~\eqref{ivp}
by an equivalent IVP with the initial value $v$
and a zero source term.
Comparing this approach with Krylov subspace evaluation of $y:=\varphi(-A)g$
shows that both Krylov subspace methods (for
$y:=\exp(-A)v - v$ and for $y:=\varphi(-A)g$) require approximately
the same number of Krylov steps.  However, the additional computational costs
for solving $Av=-g$ are not paid off.  Hence, unless the systems with $A$
can be solved at a negligible cost, computing actions
of $\varphi$ by the matrix exponential actions appears to be inefficient.
Of course, this approach can not be used for a general $g$ in case
$A$ is singular.

We remark that quite a different
and successful approach to compute the $\varphi$ matrix function via
a matrix exponential is presented in~\cite[Proposition~2.1]{Saad92}
(see also \cite{AlmohyHigham2011}).
For computing $\varphi(-tA)v$, $A\in\Rrnn$, it amounts to evaluating
the last column of $\exp(-t\widehat{A})$, where $\widehat{A}\in\Rr^{(n+1)\times(n+1)}$
is a matrix obtained by augmenting $A$ with $v$ as the last column and
zeros as the last row.  In the context of large scale problems
and Krylov subspace methods this approach has
a drawback that eventual (skew) symmetry of~$A$ is lost in~$\widehat{A}$.
We use this approach for the evaluating $u(t)=t \varphi(-tH_k) (\beta e_1)$
to solve the small scale projected IVP~\eqref{ivp1}, see Section~\ref{s:num}.
   
\section{A restarting procedure for $\varphi$ Krylov subspace evaluation}
\label{s:rst}
\subsection{The residual norm as a function of time}
\noindent
For real matrices $A$ it is not difficult to see that
property~\eqref{posdef} is equivalent to
$$
x^TAx\geqs 0, \quad \forall x\in\Rr^n,
$$
which means that the symmetric part $\frac12(A+A^T)$ of $A$ is nonnegative
definite.
Denote
\begin{equation}
\label{omega}
\omega\equiv \min_{0\ne x\in\Rr^n}\frac{x^TAx}{x^Tx}\geqs 0.  
\end{equation}
Then it holds (see, e.g.,~\cite[Section~I.2.3]{HundsdorferVerwer:book})
\begin{equation}
\label{exp_est}
\|\exp(-tA)\|\leqs e^{-t\omega},\qquad \forall t\geqs 0.  
\end{equation}
A similar estimate holds for the matrix $H_k=V_k^TAV_k$ produced by the Arnoldi process
(see~\eqref{Arn}).  Indeed, let 
\begin{equation}
\label{omegak}
\omega_k\equiv\min_{0\ne y\in\Rr^k}\frac{y^TH_ky}{y^Ty}.
\end{equation}
Then
$$
\omega_k=
\min_{0\ne y\in\Rr^k}\frac{(V_ky)^T A (V_ky)}{y^Ty}
\geqs\omega\geqs 0,  
$$
so that, similarly to~\eqref{exp_est},
\begin{equation}
\label{exp_estk}
\|\exp(-tH_k)\|\leqs e^{-t\omega_k},\qquad \forall t\geqs 0.  
\end{equation}
Exploiting the estimate~\cite[Proof of Lemma~2.4]{HochbruckOstermann2010}
$$
\|\varphi(-tA)\|\leqs\int_0^1 \|\exp(-t(1-\theta)A)\|\,d\theta,
$$
bounding the norm under the integral
by $e^{-t(1-\theta)\omega}$ and evaluating the integral, we obtain
a useful inequality
\begin{equation}
\label{phi_est}
\|\varphi(-tA)\|\leqs\varphi(-t\omega).  
\end{equation}
Then a similar relation holds also for $H_k$:
\begin{equation}
\label{phi_estk}
\|\varphi(-tH_k)\|\leqs\varphi(-t\omega_k).  
\end{equation}

\begin{proposition}
\label{prop1}
Let $t>0$, $g\in\Rr^n$ and let $A\in\Rr^{n\times n}$ be such
that~\eqref{posdef} holds.  If $y_k(t)$ is the Krylov subspace
approximation~\eqref{yk} to the vector $y(t)=t\varphi(-tA)g$ then
for its residual $r_k(t)$, defined by~\eqref{res}, \eqref{rk},
holds
\begin{equation}
  \label{rk_est}
  \|r_k(t)\|\leqs \beta t h_{k+1,k}\varphi(-t\omega_k)=
  \begin{cases}
    \dfrac{\beta}{\omega_k}h_{k+1,k}(1-e^{-t\omega_k}),\quad&\text{for }\omega_k>0,
    \\
    \beta t h_{k+1,k},\quad&\text{for }\omega_k=0,
  \end{cases}
\end{equation}
where $\omega_k$ is defined by~\eqref{omegak}.
\end{proposition}

\begin{proof}
From~\eqref{ivp1} we have $u(t)=t \varphi(-tH_k) (\beta e_1)$, so that
$$
\|u(t)\|\leqs \beta t \|\varphi(-tH_k)\| \leqs \beta t \varphi(-t\omega_k)
=\begin{cases}
  \dfrac{\beta}{\omega_k} (1-e^{-t\omega_k}),\quad&\text{for }\omega_k>0,
  \\
  \beta t ,\quad&\text{for }\omega_k=0,
\end{cases}   
$$
where~\eqref{phi_estk} is used.
Relation~\eqref{rk} implies
$$
\|r_k(t)\| = h_{k+1,k}\, |e_k^Tu(t)|\leqs h_{k+1,k} \|u(t)\|,
$$
where $h_{k+1,k}\geqs 0$ is guaranteed by the Arnoldi process.
The last relation, together with the above estimate on $\|u(t)\|$, gives
the required bound~\eqref{rk_est}.
\end{proof}

Note that \eqref{rk_est} exhibits solely the time behavior
of the residual $r_k(t)$ as a function of time for a fixed $k$.
A more detailed analysis of the residual dependence on $k$ and
on $t$ is given below in Section~\ref{s:Faber}.

\subsection{Residual-time (RT) restarting procedure}
The method for evaluating the $\varphi$ matrix functions we propose in this paper
adopts the so-called residual-time (RT) restarting strategy. It is proposed
for the matrix exponential evaluations in our recent paper~\cite{ART}.
The RT strategy is essentially based on a notion of the matrix
exponential residual defined, for an approximation $w_k(t)\approx  w(t)=\exp(-tA)v$,
as~\cite{CelledoniMoret97,DruskinGreenbaumKnizhnerman98,BGH13}
$$
-Aw_k(t)-w_k'(t).
$$
The RT restarting is based on the observation that the residual is easily
computable and the residual norm tends
to be a monotonically increasing function of time, which explains the
name ``residual-time'' restarting.
The $\varphi$ residual~\eqref{res}
exhibits the same behavior (see relation~\eqref{rk} and Proposition~\ref{prop1} above).
Therefore the RT restarting approach can be readily adopted for the $\varphi$ evaluations.

The RT restarting approach is quite simple.
We solve~\eqref{ivp}
by performing $k=1,2,\dots,k_{\max}$
steps of the Arnoldi process (see Section~\ref{s:Phi_res}).
We stop as soon as the residual norm
$\max_{s\in[0,t]}\|r_k(s)\|$ is small enough.
If, after $k_{\max}$ carried out steps,
this is not the case we find a time moment $\delta$ such that
the residual norm $\|r_{k_{\max}}(s)\|$ is sufficiently small for $s\in[0,\delta]$.
Then we set in~\eqref{ivp} $v:=y_k(\delta)$, decrease the time interval
$t:=t-\delta$ and start the Arnoldi
process for this modified IVP again.  This RT restarting for the $\varphi$
matrix function evaluations is outlined in Figure~\ref{f:rst_alg}.

\begin{figure}
\centering{\texttt{\begin{minipage}{\linewidth}
\% Given: $A\in\Rrnn$, $g\ne 0$, $v\in\Rr^n$, $t>0$, $k_{\max}$ and $\tol>0$ \\
$\conv := \mathtt{false}$, $\tol:=\tol/t$\\
while not($\conv$)                                 \\\hspace*{2em}
    $\beta:= \|g-Av\|$, $v_1 := \frac1\beta(g-Av)$         \\\hspace*{2em}
    for $k=1,\dots,k_{\max}$                         \\\hspace*{4em}
         carry out step $k$ of Arnoldi process:     \\\hspace*{4em}
         compute $v_{k+1}$ and $h_{1,k}$, \dots, $h_{k+1,k}$ \\\hspace*{4em}
         if $\max_{s\in[0,t]}\|r_k(s)\|\leqs\tol$ then \hfill $(*)$      \\\hspace*{6em}
            $\conv:=\mathtt{true}$           \\\hspace*{6em}
            break (leave the for loop)          \\\hspace*{4em}
        elseif $k=k_{\max}$                                           \\\hspace*{6em}
            \% -------- restart after $k_{\max}$ steps                \\\hspace*{6em}
            find largest $\delta$ such that
            $\max_{s\in[0,\delta]}\|r_k(s)\|\leqs\tol$
               \hfill $(\circ)$ \\\hspace*{6em}
            $u :=\delta\varphi(-\delta H_k)e_1$            \\\hspace*{6em}
            $v := v + V_k(\beta u)$                                     \\\hspace*{6em}
            $t := t - \delta$                                     \\\hspace*{6em}
            break (leave the for loop)                                \\\hspace*{4em}
        end if                                                         \\\hspace*{2em} 
   end for                                                           \\
end while                                                              \\
$u :=t\varphi(-t H_k)e_1$, 
$y_k := v + V_k(\beta u)$
\end{minipage}}}
\caption{Description of the RT (residual--time) restarting algorithm
  for the $\varphi$ matrix function evaluations.
  The algorithm computes Krylov subspace approximation $y_k(t)\approx v+\varphi(-tA)(g-Av)$
  such that for its residual $r_k(t)$ holds $\max_{s\in[0,t]}\|r_k(s)\|\leqslant\tol$.}
\label{f:rst_alg}
\end{figure}

To monitor the values of the residual norm in the algorithm
(see the algorithm lines marked with $(*)$, $(\circ)$), we
compute the exact solution $u(s)$ of the projected IVP~\eqref{ivp1}
at time moments $s=\Delta t$, $2\Delta t$, $3\Delta t$, \dots.  This
is done by recursion
\begin{equation}
\label{u(s)}
\begin{aligned}
  u(s+\Delta t) &= E_ku(s) + \Delta t d_k,
  \\
   & E_k=\exp(-\Delta t H_k),\quad d_k=\beta\varphi(-\Delta t H_k)e_1,
\end{aligned}
\end{equation}
where the matrix $E_k$ and vector $d_k$ are computed once.
The value of $\Delta t$ is chosen in lines~$(*)$ and~$(\circ)$
differently:
\begin{equation}
\label{Dt}  
\Delta t :=
\begin{cases}
  t/6, \quad & \text{to estimate }\displaystyle\max_{s\in[0,t]}\|r_k(s)\|,
  \text{ line }(*),
  \\
  \dfrac{t}{2^p100}, \quad & \text{with smallest }p=0,1,\dots
  \text{ such}
  \\
  &\text{that }\|r_k(\Delta t)\|\leqs\tol,
  \text{ line }(\circ).
\end{cases}
\end{equation}
For estimating $\max_{s\in[0,t]}\|r_k(s)\|$ in line~$(*)$ of the algorithm,
$\Delta t$ does not have to be very small.  
This is because $\|r_k(s)\|$ usually achieves its maximum at the end point $s=t$
and the value $\|r_k(t)\|$ determines the quality of the approximation
$y_k(t)\approx y(t)$.
We emphasize that we compute the values of $u(s)$ exactly, and the computation
accuracy does not depend on $\Delta t$ as is the case in a regular
time integration scheme.

To determine the restart interval length $\delta$
(line~$(\circ)$ in Figure~\ref{f:rst_alg}),
we proceed according to formula~\eqref{Dt}.
More specifically, 
$\Delta t$ is initially set to $t/100$.
If $\|r_k(\Delta t)\|\leqs\tol$,
we trace the values $\|r_k(\Delta t)\|$,
$\|r_k(2\Delta t)\|$, \dots
by recursion~\eqref{u(s)}, 
until they exceed the tolerance value.
Then we set $\delta$ to the largest $p\Delta t$
for which $\|r_k(p\Delta t)\|\leqs\tol$.
Otherwise, if $\|r_k(\Delta t)\|>\tol$
for $\Delta t=t/100$, we keep on halving $\Delta t$ until
$\|r_k(\Delta t)\|\leqs\tol$ and then set $\delta:=\Delta t$.

\input{L_content.tex}

\section{Convergence of the restarted method}
\label{s:conv}

\begin{proposition}
  Let the RT restarted Krylov subspace method be used to compute an approximation
  to the solution $y(t)$, $t<\infty$, of IVP~\eqref{ivp} and
  let $\tol$ be the residual tolerance used at each restart.
  Then for any restart length $k_{\max}$ the RT restarted Krylov subspace method
  stops after a finite number $J$ of restarts and 
  the error of its approximate solution $y_k(t)$ 
  is bounded by $t\cdot\tol$,
  $$
  \|y(t)-y_k(t)\|\leqs t\cdot\tol,
  $$
  with $t$ being the length of the time interval on which problem~\eqref{ivp}
  is solved.
\end{proposition}

\begin{proof}
  Based on Propositions~\ref{prop1}, \ref{P2}, \ref{P4},
  we know that at each restart $j$
  a $\delta_j>0$ can be found such that the residual norm does not exceed
  any tolerance $\tol$ on the time interval of length $\delta_j$.
  Hence, since the time interval length $t$ is finite and is reduced
  at each restart by $\delta_j$, a finite number $J$ of
  restarts will suffice.  
  The time interval after $j$ restarts is reduced from
  $[0,t]$ to
  $[0,t-(\delta_1+\dots+\delta_j)]$ and $\sum_{j=1}^J\delta_j=t$.

  To simplify the notation in the proof, 
  by $y_j(s)$, $s\in[0,\delta_j]$ we denote
  the Krylov subspace solution which is obtained, with $k\leqs k_{\max}$
  steps of the Arnoldi process, at restart $j$
  ($k< k_{\max}$ is possible only for $j=J$).
  We also denote the residual of $y_j(s)$ by $r_j(s)$.
  With this notation, $y_j(s)$ approximates the exact solution
  $y(\sum_{i=1}^{j-1}\delta_i + s)$
  for $s\in[0,\delta_j]$ with residual $r_j(s)$ such that
  $$
  \max_{s\in[0,\delta_j]}\|r_j(s)\|\leqs\tol.
  $$
  Then the approximate solution at the first restart is
  $$
  y_1(\delta_1)= v + \delta_1\varphi(-\delta_1A)(g-Av) + \epsi_1(\delta_1),
  $$
  where $\epsi_1(s)$ is the error of the approximation $y_1(s)\approx y(s)$.
  Using~\eqref{err_bnd} we can bound $\|\epsi_1\|$ as
  $$
  \|\epsi_1(\delta_1)\|\leqs \delta_1\varphi(-\delta_1\omega)\max_{s\in[0,\delta_1]}\|r_1(s)\|
  \leqs \delta_1\varphi(-\delta_1\omega)\cdot\tol
  \leqs \delta_1\cdot\tol,
  $$
  where $\omega\geqs 0$ is the constant from~\eqref{omega}.
  To further simplify the notation, let us denote $\epsi_j=\epsi_j(\delta_j)$
  and $y_j=y_j(\delta_j)$.  Then at restart~2 we have
  $$
  \begin{aligned}
    y_2 &= y_1+\epsi_1 + \delta_2\varphi(-\delta_2A)(g-A(y_1+\epsi_1)) + \epsi_2\\
    &=\underbrace{y_1 + \delta_2\varphi(-\delta_2A)(g-Ay_1)}_{\text{exact solution}}
    + \, \underbrace{\delta_2\varphi(-\delta_2A)(-A\epsi_1) +\epsi_1+ \epsi_2}_%
    {\varepsilon_2},
  \end{aligned}
  $$
  where we denote all the error terms by $\varepsilon_2$.  Now note that
  $$
  \delta_2\varphi(-\delta_2A)(-A\epsi_1) +\epsi_1 = \exp(-\delta_2A)\epsi_1,
  $$
  and, hence,
  $$
  \varepsilon_2=\exp(-\delta_2A)\epsi_1 + \epsi_2,
  \quad \|\varepsilon_2\|\leqs
  \|\epsi_1 + \epsi_2\|\leqs(\delta_1+\delta_2)\cdot\tol.
  $$
  Here we bound $\epsi_2$ in the same way as $\epsi_1$ and use
  estimate~\eqref{exp_est}.
  Continuing in the same way, we obtain at the last restart $J$
  $$
  \begin{aligned}
  \varepsilon_J &= \exp(-(t-\delta_1)A)\epsi_1 +
  \exp(-(t-\delta_1-\delta_2)A)\epsi_2 + \dots \\
  &\phantom{=} +  \exp\left(-(t-\sum_{j=1}^{J-1}\delta_j)A\right)\epsi_{J-1} + \epsi_J,
  \\
  \|\varepsilon_J\| & \leqs \|\epsi_1\| + \dots + \|\epsi_J\|
  \leqs (\delta_1+\dots+\delta_J)\cdot\tol = t\cdot\tol. 
  \end{aligned}
  $$
  \end{proof}

\begin{figure}
\centerline{%
\includegraphics[width=0.4\textwidth]{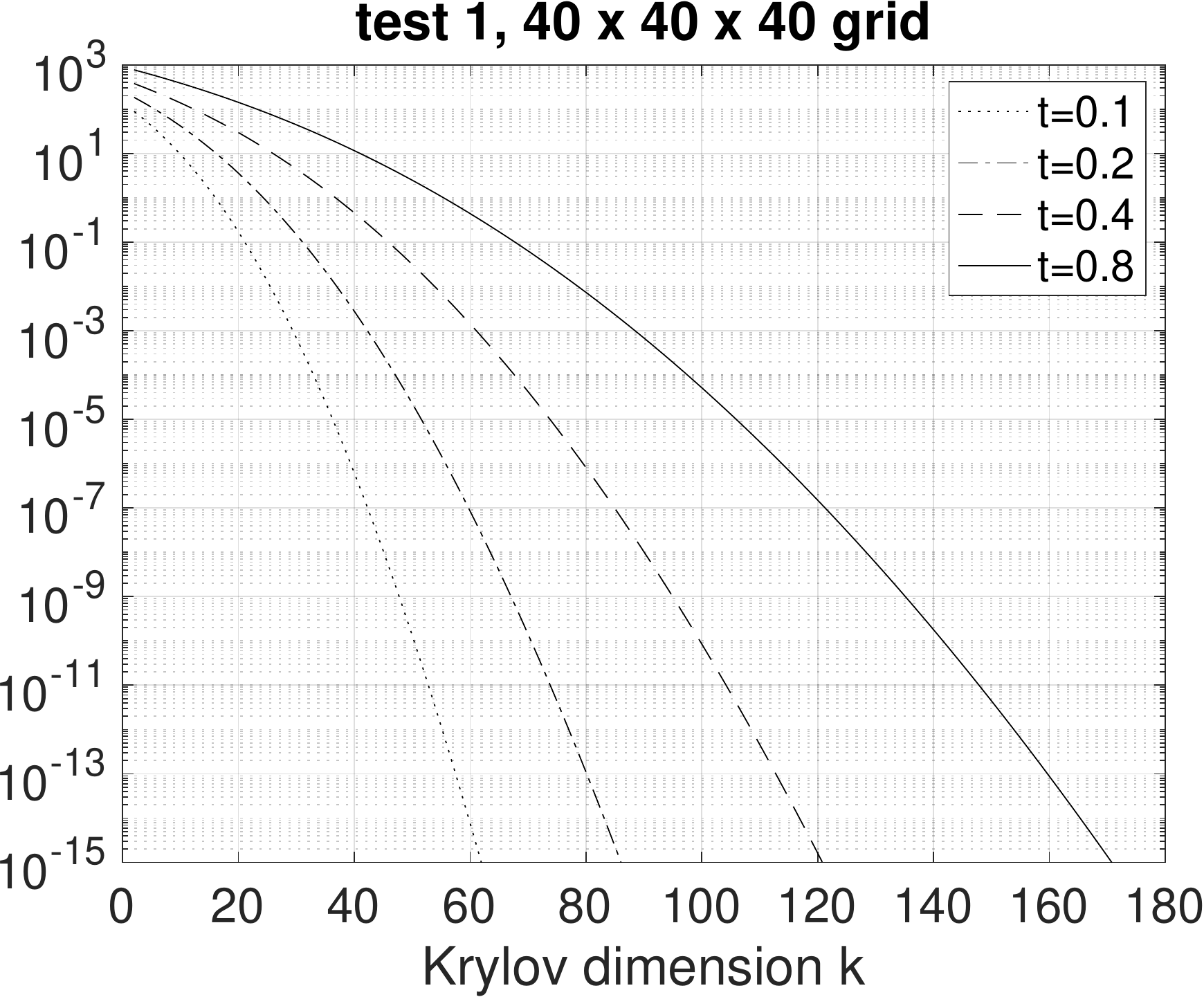}}
\caption{Residual upper bound~\eqref{resP2} for
$c=501.1359$ (the first test problem, $40\times 40\times 40$ grid)
and different time interval lengths $t$.}
\label{f:est_tt}
\end{figure}

The number of restarts $J$ appearing in Proposition~\ref{s:conv},
can be estimated from the residual upper
bounds~\eqref{resP2},\eqref{resP3},\eqref{resP4}.
This can be done, for example, using plots as shown in Figure~\ref{f:est_tt}.
There, for the test problem described in Section~\ref{s:t1},
we plot the bound~\eqref{resP2} for different time interval lengths $t$.
For instance, if $k_{\max}\approx 30$ and the requested tolerance
is in the range $[10^{-5},10^{-3}]$, then
the restart interval length $\delta$ should be approximately $0.2$
(see the dotted curve in Figure~\ref{f:est_tt}).
Hence, for the time interval length $t=1000$ used in the test,
$J\approx 1000/0.2=5000$ restarts are expected.
Since the bound~\eqref{resP2} is not sharp, this should be
seen as a safe, pessimistic estimate on $J$.
Indeed, as we see from Table~\ref{t:p4}, the actual number of restarts
is roughly four times smaller.

\section{Numerical experiments}
\label{s:num}
In the experiments presented here we test our method
to compute actions of $\varphi(-tA)$ implemented
with the proposed RT restarting.  We compare our code \phiRT{}
with two other well known Krylov subspace codes for evaluating
$\varphi(-tA)$, namely
the \phiv{} function of the EXPOKIT package~\cite{EXPOKIT}
and the \phip{} function presented in~\cite{NiesenWright12}
and publicly
available online\footnote{See \url{http://www1.maths.leeds.ac.uk/~jitse/software.html}}.
We do not include the recently proposed
code \phipm{}~\cite{NiesenWright12} in the comparisons because this code
uses a variable Krylov dimension up to~$100$.  Using such a high Krylov
dimension in the large scale problems is not always feasible,
and in the tests we restrict the Krylov dimension for all the
three solvers by~$30$.

All the tests are carried out on a Linux laptop with four AMD Ryzen~3 3200U CPUs
(clock rate 2.6~GHz) and 8~Gb memory, in Octave (version~4.2.2)
with no parallelization.
Our implementation of the Arnoldi and Lanczos processes does not include
Krylov basis vector reorthogonalization and no serious
orthogonality loss has been observed.
In the numerical tests the Krylov subspace matrices are
computed by the Lanczos process as soon as $A=A^T$, otherwise
the Arnoldi process is employed.

To calculate matrix-vector products $\varphi(-tH_k)x$
for the small projected Krylov subspace matrix $H_k$
in our algorithm (cf.~\eqref{u(s)}),
we employ relation~\cite[Proposition~2.1]{Saad92}
$$
\exp(
\begin{bmatrix}
-tH_k & x \\ 0 & 0  
\end{bmatrix})
= 
\begin{bmatrix}
\exp(-tH_k) & \varphi(-tH_k)x \\ 0 & 1  
\end{bmatrix},
$$
computing the matrix exponential at the left hand side and,
thus, reducing computations with $\varphi$ to
computations with the matrix exponential
(see also~\cite{AlmohyHigham2011}).
To compute the matrix exponentials $\exp(-tH_k)$ here and
in~\eqref{u(s)} the standard built-in function \texttt{expm} of Octave
is used.
It is based on Ward's scaling and squaring method combined with
diagonal Pad\'e approximation and is certainly not the best method available,
see~\cite{19ways,HighamAlmohy2010}.
For the moderate tolerance range $[10^{-7},10^{-3}]$
we aim in the tests, we have observed no accuracy problems
with this Octave implementation.
We have chosen to stick to the built-in function \texttt{expm} to have our
Octave code fully Matlab portable. Note that Matlab's implementation
of \texttt{expm} \cite{AlmohyHigham2010b} is seen as more reliable but is not
publicly available.
For higher accuracy requirements we could switch, for instance,
to the \texttt{phipade} function of the EXPINT package~\cite{EXPINT}.

In all the tests the reported achieved error
is measured for each of the three Krylov subspace solvers as
$$
\frac{\|y_k(t)-y_{\mathrm{ref}}(t)\|}{\|y_{\mathrm{ref}}(t)\|},
$$
where $y_k(t)$ is the solution of the Krylov subspace solver at final time $t$
and $y_{\mathrm{ref}}(t)$ is a reference solution computed with a high tolerance
(reported for each test below) by the \phiv{} function.
Since the reference
solution $y_{\mathrm{ref}}(t)$ is affected by the same space discretization error as
$y_k(t)$, 
$\|y_k(t)-y_{\mathrm{ref}}(t)\|$ is a good indicator of the time error.

\subsection{High-order discretized anisotropic elliptic operator}
\label{s:t1}
In this test we solve IVP~\eqref{ivp} with a symmetric matrix $A$.
It is a 3D (three-dimensional) anisotropic elliptic operator
$10^4\partial_{xx} + 10^2\partial_{yy} + \partial_{zz}$
with periodic boundary conditions defined in the spatial domain $[0,1]\times[0,1]\times[0,1]$
and discretized by a fourth order 
accurate finite volume discretization described in~\cite{VerstappenVeldman2003}.
The source vector $g$ contains the grid values of either
of the functions
$$
\begin{aligned}
& \sin(2\pi x)\sin(2\pi y)\sin(2\pi z) +
  x(1-x) y(1-y) z(1-z),
\\
& \mathrm{e}^{ -500( (x-0.5)^2 + (y-0.5)^2 + (z-0.5)^2 ))} +
x(1-x) y(1-y) z(1-z).
\end{aligned}
$$
We use two different source vectors $g$ to make sure that Krylov subspace
methods do not profit from a particular choice of the source vector.
In this test the initial value vector $v$ in~\eqref{ivp} is set to zero.

In the test we use uniform $40\times40\times40$ and $60\times60\times60$ grids,
for which the problem size is $n=64\,000$ and $n=216\,000$,
respectively.
For these grids the matrix $A$ has its norms
$\|A\|_1=\|A\|_{\infty}$ approximately equal to $1002.27$ and 
$197.98$, respectively.  These values give the estimations $\|A\|_1/2$
of $c$ used in Figure~\ref{estsymm}.
The final time is $t=1000$ and
the reference solution is computed by
function \phiv{} with tolerance $\tol=\texttt{1e-11}$.
%

The results for this test are presented in Table~\ref{t:p4}
and in Figure~\ref{f:acc_p4},
where for each of the solvers 
the values of achieved accuracy, the number of matrix-vector multiplications
(matvecs) and the CPU times are given.
These values are reported for the same tolerance {\tt1e-6}.
Exploring performance
of the solvers for different tolerance values leads to plots in
Figure~\ref{f:acc_p4}.  As can be seen in the plots, on the coarser grid
function \phiv{} fails to produce results with an accuracy less stringent
than $\approx 10^{-7}$.  Its least accurate result (corresponding to
the upper left end of the dashed curve) is obtained for
tolerance {\tt5e-2}. An attempt to run the code with a less stringent tolerance
{\tt1e-1} yields an error message.  Furthermore,
the \phip{} solver does not produce
the results within the required accuracy range $[10^{-7},10^{-3}]$
and appears to be more expensive for this test than our \phiRT. 
For both spatial grids the \phip{} solver exhibits
an abrupt change in the delivered accuracy with gradually varying
requested tolerance, indicated in the Figure by
segments \textsf{AB} and \textsf{CD}.  Thus, we see that
both \phiv{} and \phip{} solvers have difficulties
with delivering results within the desired accuracy
range $[10^{-7},10^{-3}]$.

\begin{table}
\caption{Results for the fourth order finite-volume discretized anisotropic heat equation}
\label{t:p4}
\centering{\begin{tabular}{lcc}
\hline\hline
method(Krylov dim.),       & delivered & matvecs /   \\
requested tolerance        &   error   & CPU time, s \\
\hline
\multicolumn{3}{c}{$40\times 40\times 40$ grid, sine $g$}
\\
\phiv($10$), {\tt1e-06}    & {\tt3.05e-12} & 185472 / 1380 \\
\phip($10$), {\tt1e-06}    & {\tt7.75e-14} & 170420 / 1282 \\
\phiRT($10$), {\tt1e-06}   & {\tt1.37e-09} & 101487 / 824  \\
\phiv($30$), {\tt1e-06}    & {\tt6.07e-13} & 61632 / 482 \\
\phip($30$), {\tt1e-06}    & {\tt6.89e-14} & 79710 / 545 \\
\phiRT($30$), {\tt1e-06}   & {\tt7.78e-09} & 33837 / 238 \\
\hline
\multicolumn{3}{c}{$40\times 40\times 40$ grid, Gaussian $g$}
\\
\phiv($30$),  {\tt1e-06} & {\tt1.33e-10}  & 63520 / 565 \\
\phip($30$),  {\tt1e-06} & {\tt9.98e-14}  & 80910 / 572 \\
\phiRT($30$), {\tt1e-06} & {\tt8.56e-09}  & 36117 / 324 \\
\hline
\multicolumn{3}{c}{$60\times 60\times 60$ grid, sine $g$}
\\
\phiv($30$),  {\tt1e-06}  & {\tt2.40e-11}  & 16800 / 692  \\
\phip($30$),  {\tt1e-06}  & {\tt2.91e-13}  & 17640 / 525  \\
\phiRT($30$), {\tt1e-06}  & {\tt1.59e-08}  & 10191 / 258  \\
\hline
\end{tabular}}
\end{table}

\begin{figure}
\begin{center}    
\includegraphics[height=0.37\textwidth]{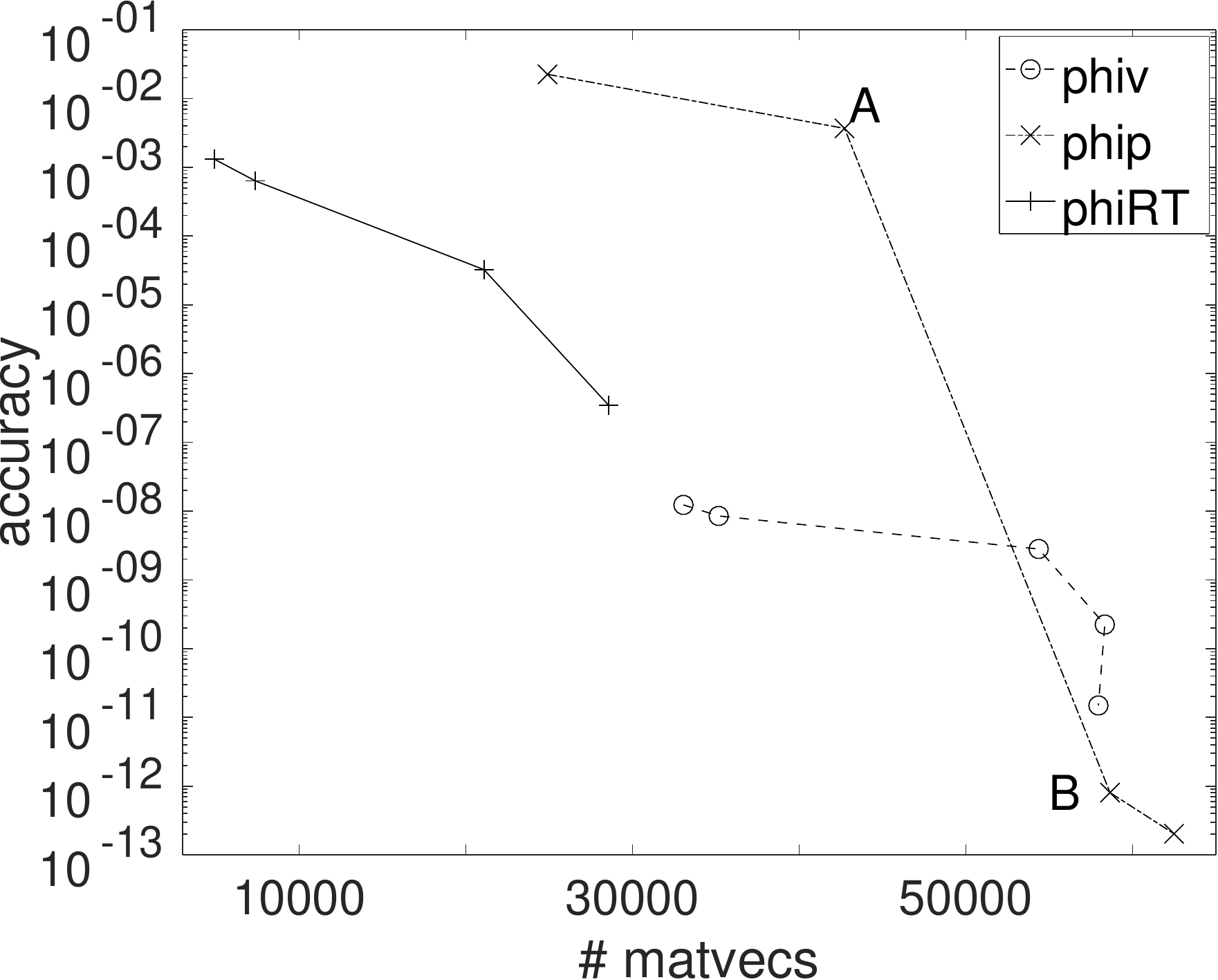}\hfill
\includegraphics[height=0.37\textwidth]{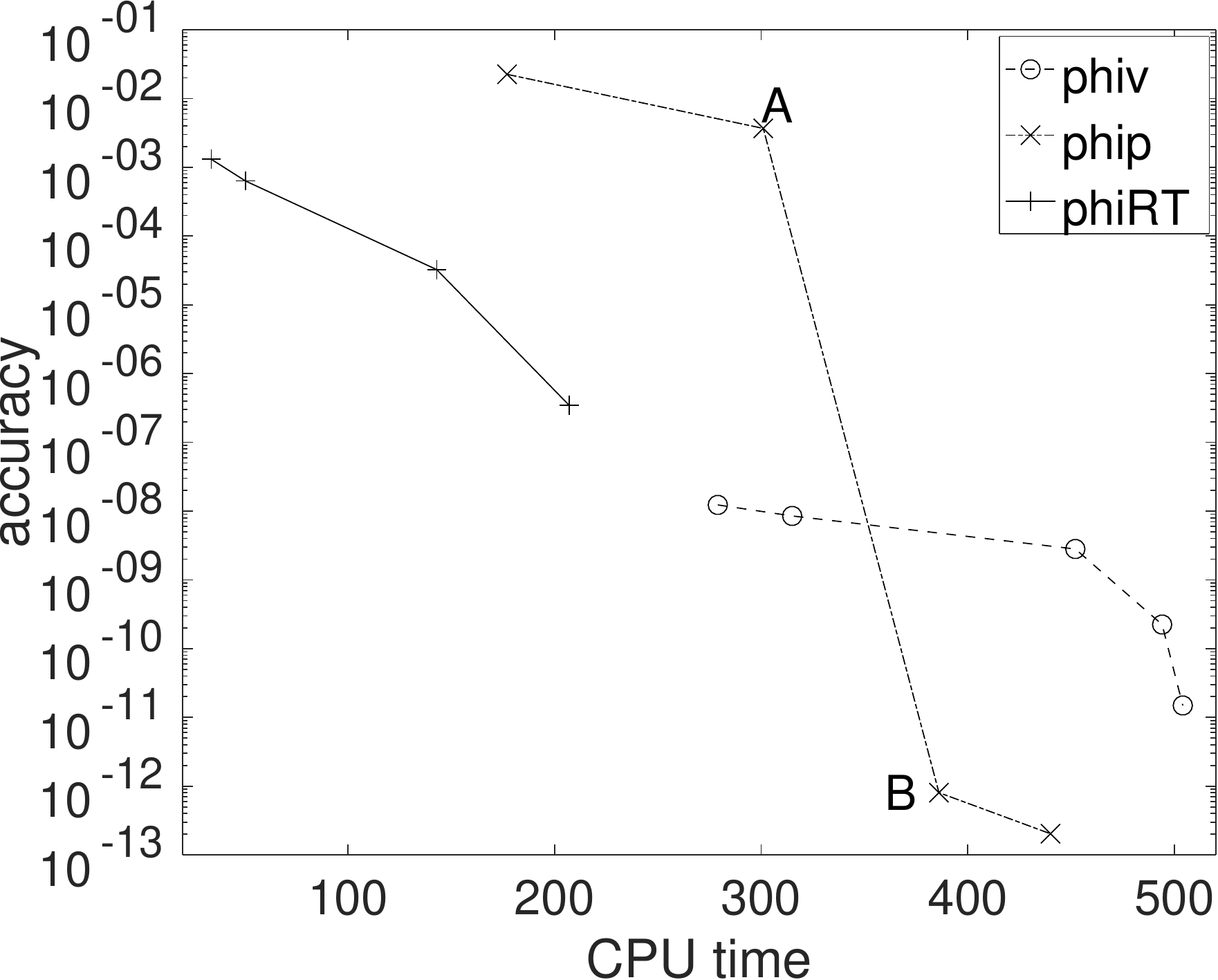}  
\\[2ex]
\includegraphics[height=0.37\textwidth]{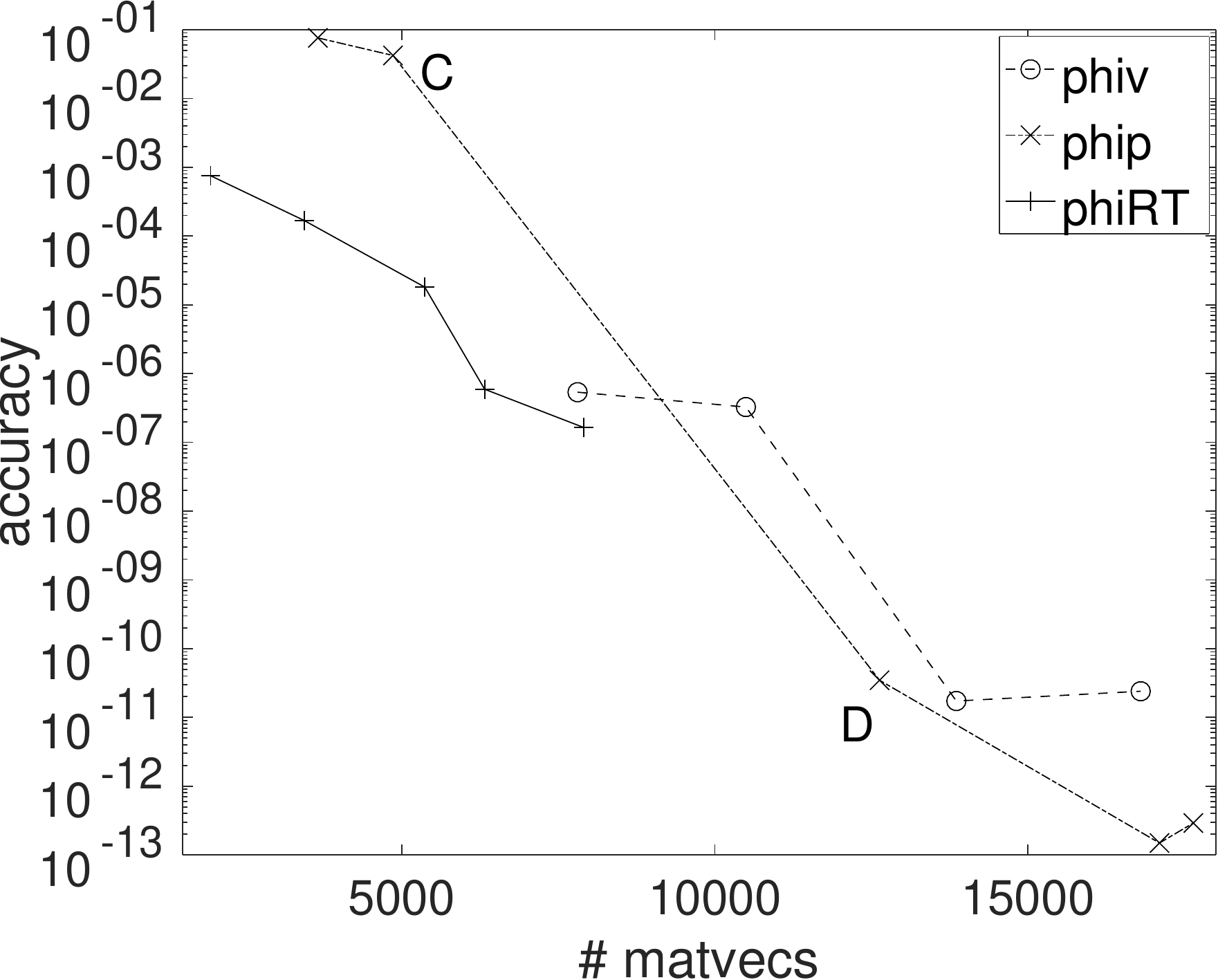}\hfill
\includegraphics[height=0.37\textwidth]{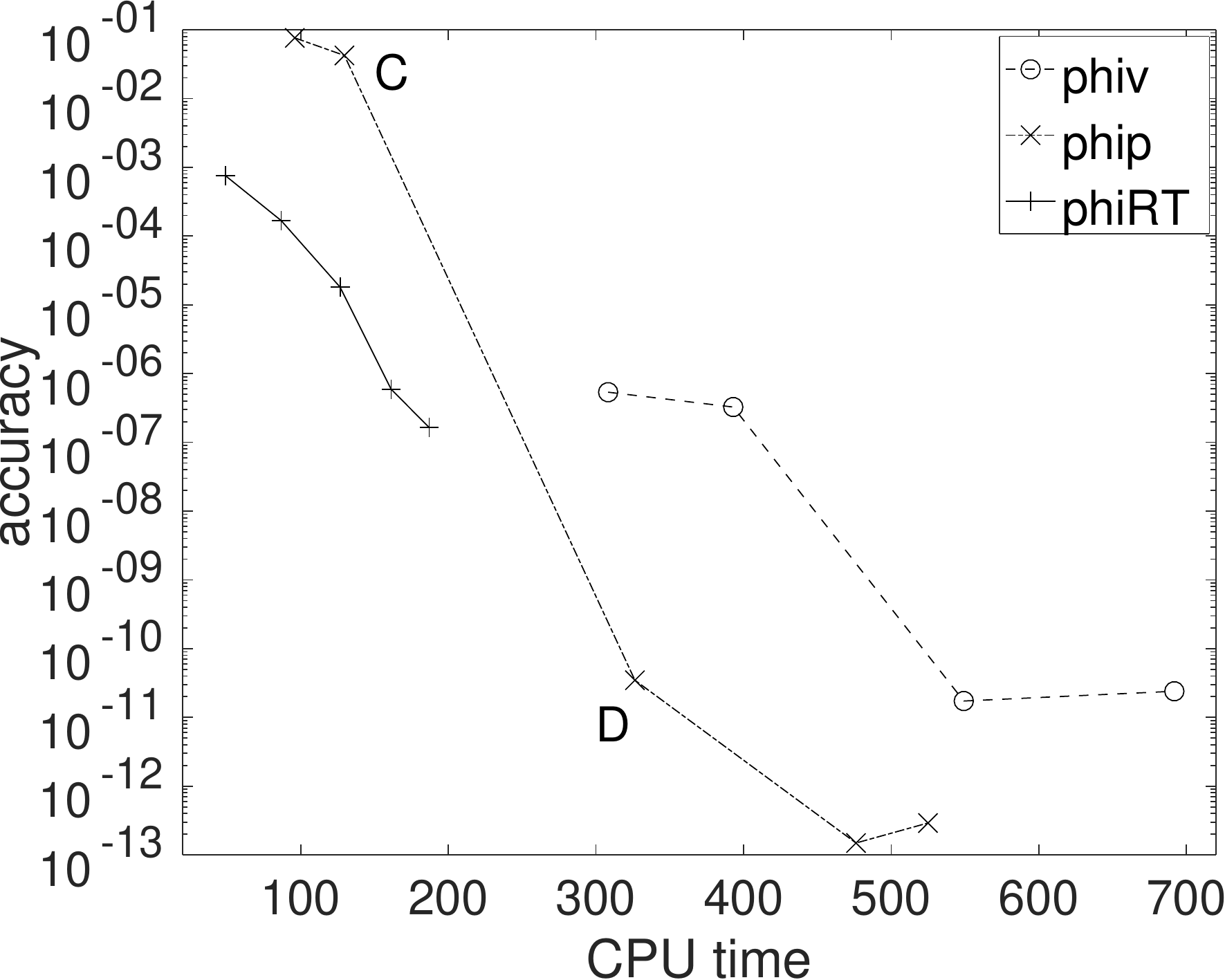}  
\end{center}
\caption{The fourth order finite-volume discretized anisotropic heat equation.
Relative error versus number
of the matrix-vector products (left)
and versus the CPU time (right) for the \phiv{} (dashed line), \phip (dash-dotted line)
and \phiRT{} (solid line).
The restart length is~30 for all the three solvers.
The upper left end of the \phiv{} curve at the top plots is obtained
for tolerance {\tt5e-02} and 
running \phiv{} with tolerance {\tt1e-01} fails with an error message.
The points \textsf{A}, \textsf{B}, \textsf{C}, and \textsf{D} on the \phip{}
curve are obtained for the tolerances
{\tt5e-04}, {\tt1e-04}, {\tt3e-03}, and {\tt1e-03}, respectively.}
\label{f:acc_p4}
\end{figure}


\begin{figure}
\centering{%
\includegraphics[width=0.4\textwidth]{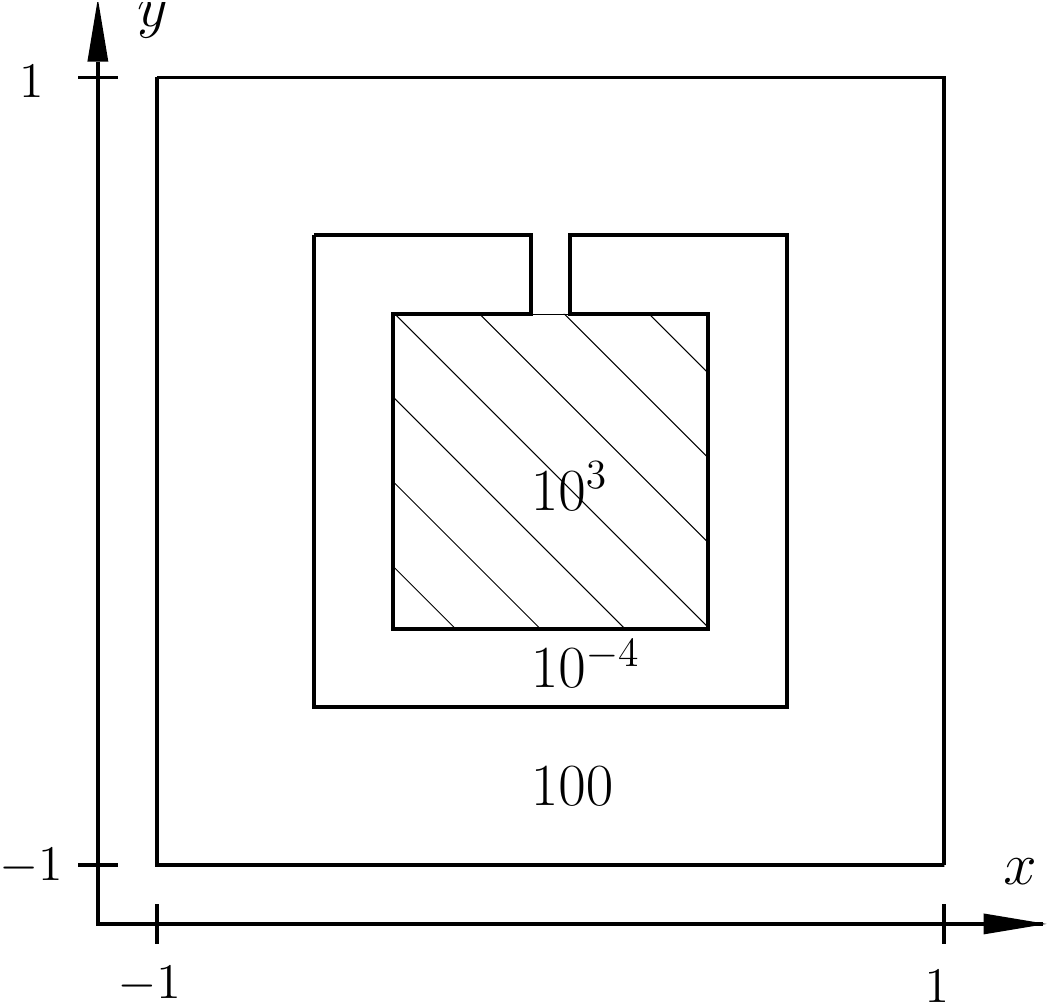}\hfill
\includegraphics[width=0.4\textwidth]{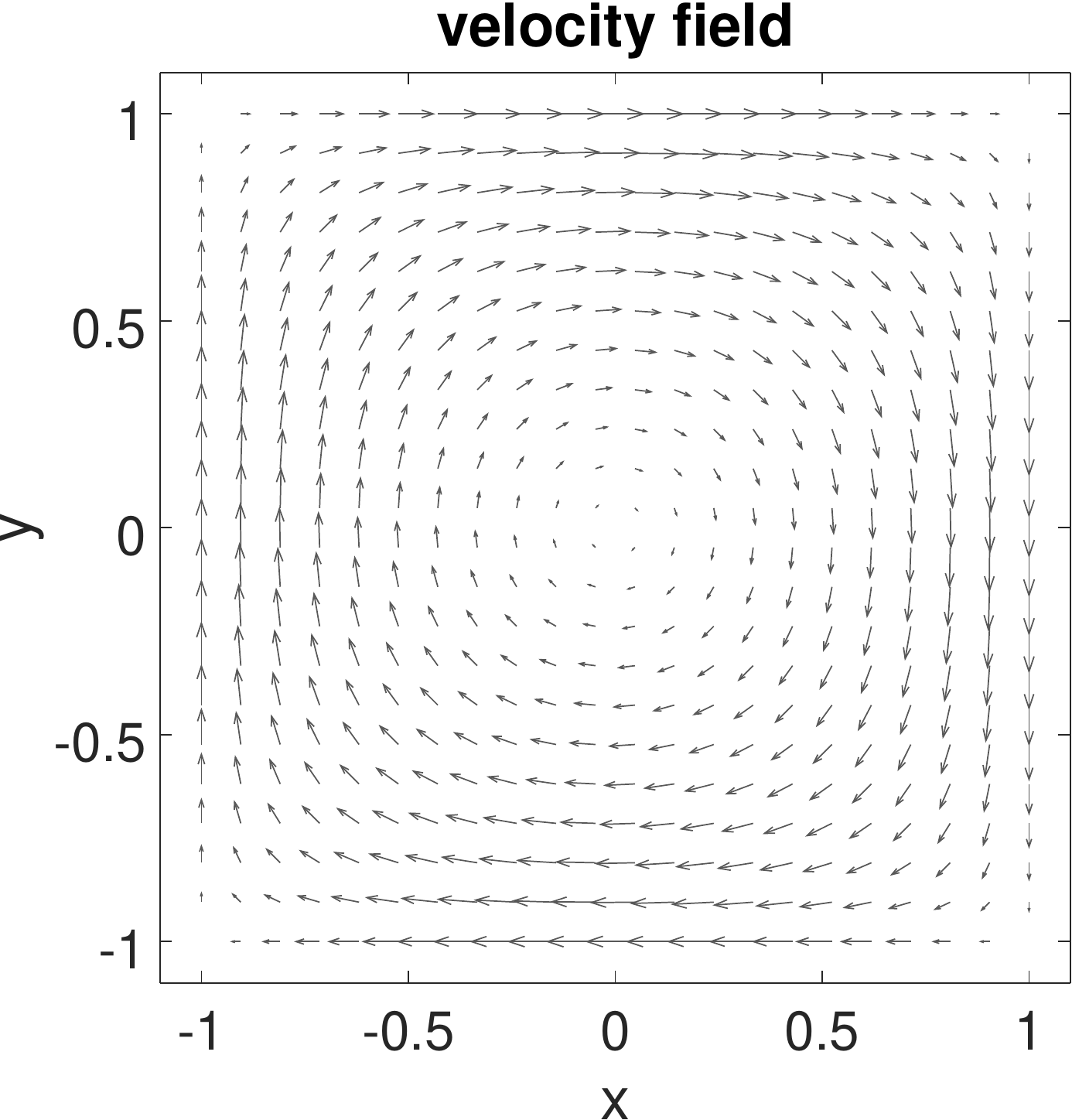}}    
\caption{Diffusion coefficients $D_{1,2}(x,y)$ (left) and velocity field
$[ v_1(x,y), v_2(x,y)]$ (right). The region where $D_{1,2}=10^3$
occupies the square $[-0.4,0.4]\times[-0.4,0.4]$ and is surrounded
by the wall where $D_{1,2}=10^{-4}$ which is $0.2$ wide and has a
$0.1$ broad slit.}
\label{f:cd}
\end{figure}

\subsection{Convection--diffusion discretized by finite differences, complex diffusion}
\label{s:t2}
We now consider a nonsymmetric matrix $A$ resulting from
a standard five point central-difference discretization
of the following convection--diffusion operator:
$$
L[u]=-(D_1u_x)_x-(D_2u_y)_y + \Pe{}\,\left(
  \frac12(v_1u_x + v_2u_y) + \frac12((v_1u)_x + (v_2u)_y) \right),
$$
where $(x,y)\in [-1,1]\times [-1,1]$ 
and the homogeneous Dirichlet conditions are imposed on the boundary.
Here the first derivatives representing the convective terms
are written in a special way.
If the velocity field is divergence-free, i.e.,
if $(v_1)_x + (v_2)_y=0$, then the central-difference
discretization of the convective terms yields a skew-symmetric
matrix~\cite{Krukier79}.  This physically desirable
property~\cite{VerstappenVeldman2003}
guarantees that the first derivatives do not contribute
to the symmetric part of the matrix which is the case
for diffusive upwind discretizations.

The diffusion coefficients $D_{1,2}(x,y)$ are defined as
shown at the left plot in Figure~\ref{f:cd}.  This choice
is inspired by~\cite[test~4]{BiCGStab}.
The velocity field, taken as
$$
v_1(x,y) = y(1-x^2), \qquad
v_2(x,y) = x(y^2-1),
$$
is the recirculating wind adopted from the IFISS
package~\cite{IFISS}, see the right plot in Figure~\ref{f:cd}.
The initial value vector $v$, cf.~\eqref{ivp},
has all its entries set to $0.01$.
For $v=0$ the solution may become
unstable due to appearance of small negative solution values
(recall that central difference discretizations are used
for both diffusion and convection terms).
The components of the source vector $g$ are the values of function
$$
1000 \mathrm{e}^{-100(x^2 +y^2)}
$$
on the finite difference grid.
Uniform $202\times 202$, $402\times 402$
and $802\times 802$ grids are used in the experiments,
so that the problem size is $n=40\,000$,
$n=160\,000$ and $n=640\,000$, respectively.
On the first two grids the Peclet number $\Pe$ is
set to $10$, whereas $\Pe=20$ on the finest grid.
For these Peclet values and all the three grids
the ratio of the norms of the skew-symmetric
and symmetric parts of the matrices, measured
as the maximal row and column norms, is of order $10^{-5}$,
i.e., the matrices are weakly nonsymmetric.
Solution samples of this test problem can be
seen in Figure~\ref{f:cd_sol}.
The reference solution is computed by the \phiv{} function
run with tolerance {\tt1e-14}.

The results of the comparison runs are presented in Table~\ref{t:cd}.
To see clearly the dependence of the attained accuracy on the
computational work and CPU time, we also present plots
in Figure~\ref{f:cd_acc}.  We see that \phiv{} tends to deliver
a much more accurate solution than is requested by given tolerance.
We also see that for the coarsest $202\times 202$ grid our solver is more efficient
than \phip{} in terms of matrix--vector multiplications but
requires more CPU~time.  This effect disappears on the finer grids.
There, for the moderate tolerance values we are interested in,
\phiRT{} clearly outperforms the other solvers in terms of CPU~time
and matrix--vector multiplications.

\begin{figure}
\centering{%
\includegraphics[width=0.4\textwidth]{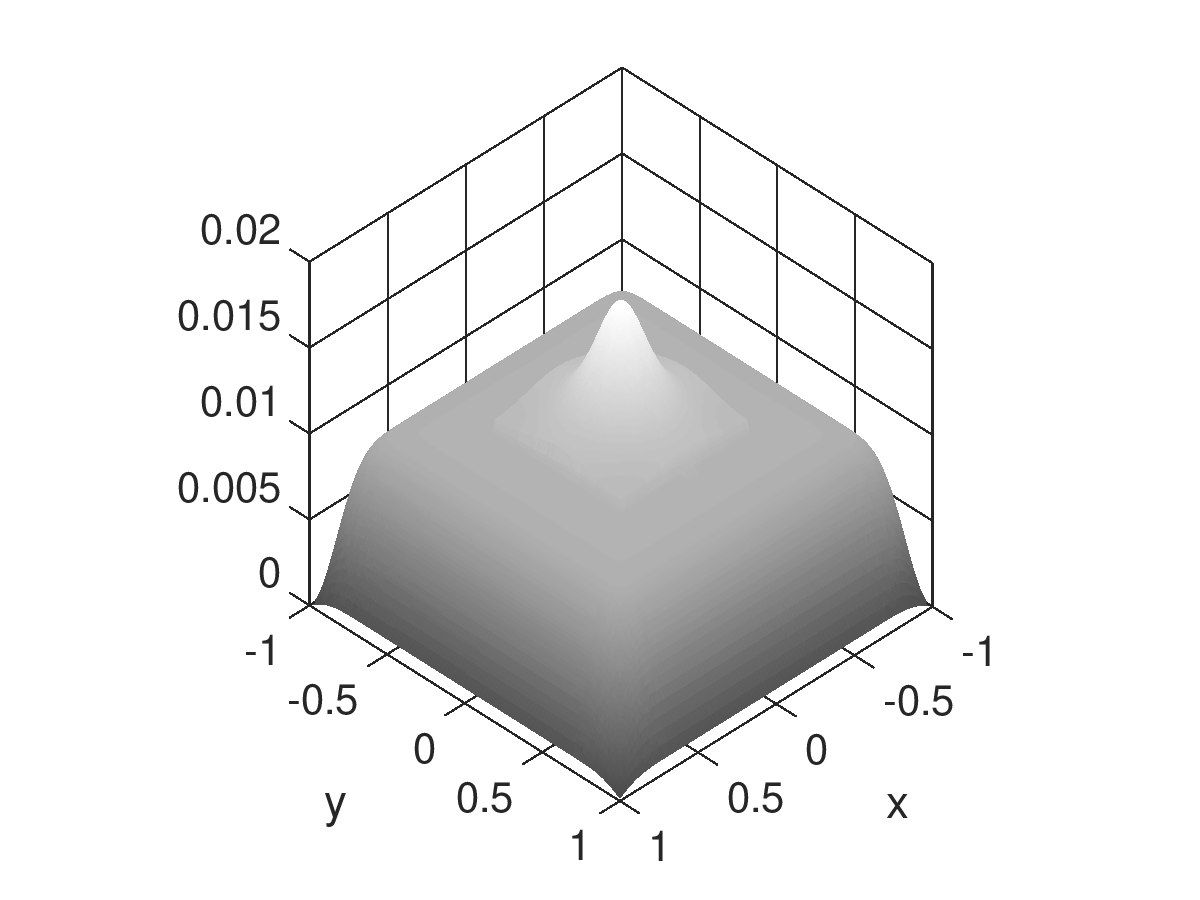}\hfill
\includegraphics[width=0.4\textwidth]{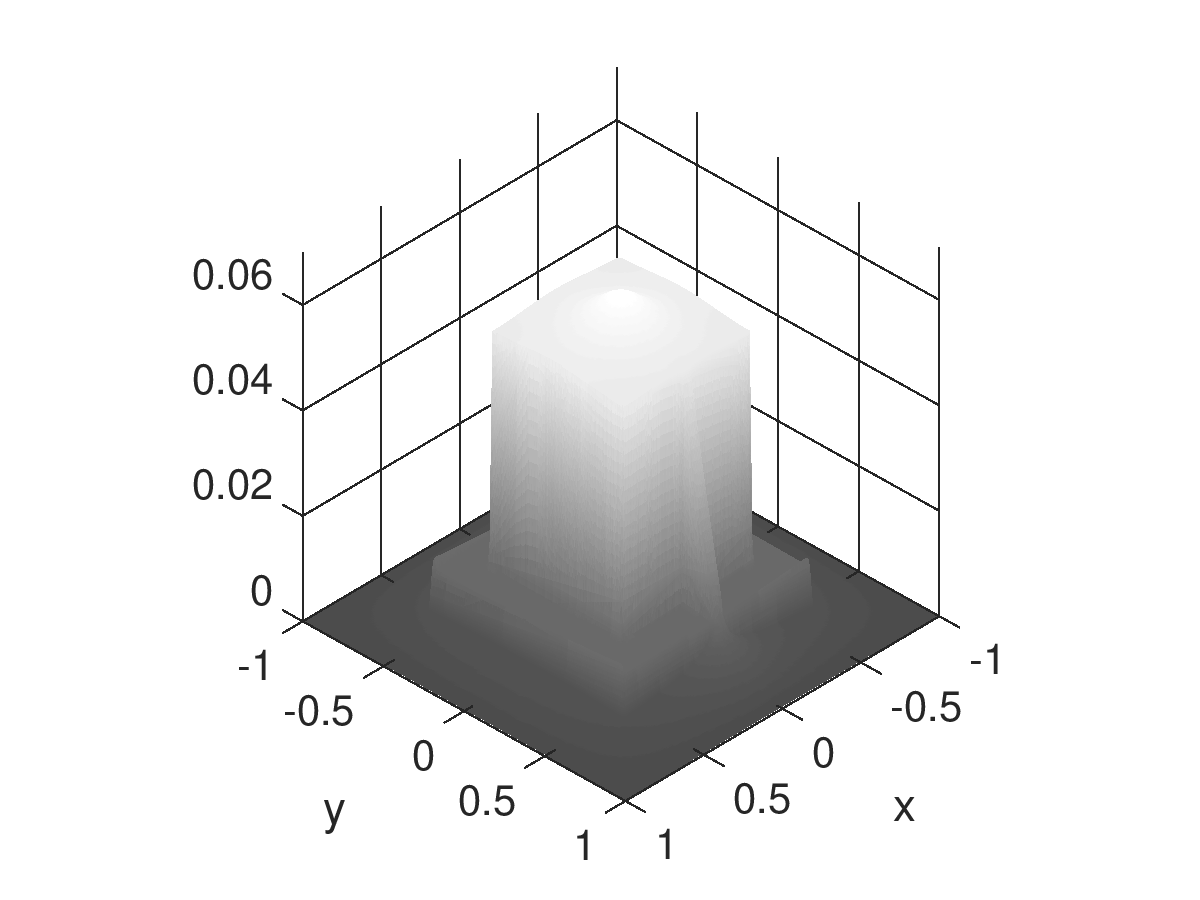}}    
\caption{Solution of finite-difference discretized convection--diffusion
problem at time $t=5\times10^{-5}$ (left) and $t=10^{-3}$ (right),
$202\times 202$ grid, $\Pe=10$. 
At the right plot a correspondence to the values
of the diffusion coefficients $D_{1,2}$ is clearly visible, with
a leakage through the slit, see Figure~\ref{f:cd}.}
\label{f:cd_sol}
\end{figure}

\begin{table}
\caption{Results for the finite-difference discretized convection--diffusion
test problem.}
\label{t:cd}
\centerline{\begin{tabular}{lcc}
\hline\hline
method(Krylov dim.),  & delivered & matvecs /   \\
requested tolerance   &    error  & CPU time, s \\
\hline
\multicolumn{3}{c}{$202\times 202$ grid, $\Pe=10$, $t=10^{-3}$}\\
\phiv($30$), {\tt5e-1}  & {\tt1.73e-08} &  6016 / 16.6   \\
\phiv($30$), {\tt1e-1}  & {\tt1.66e-09} &  6528 / 18.2   \\
\phip($30$), {\tt5e-1}  & {\tt1.67e-04} &  2700 / 8.1     \\
\phip($30$), {\tt1e-1}  & {\tt8.16e-06} &  2850 / 8.6     \\
\phiRT($30$), {\tt1e-2} & {\tt1.36e-04} &  2368 / 9.2    \\
\phiRT($30$), {\tt5e-3} & {\tt3.80e-05} &  2399 / 9.6     \\
\hline
\multicolumn{3}{c}{$402\times 402$ grid, $\Pe=10$, $t=10^{-3}$}\\
\phiv($30$),  {\tt5e-1} & {\tt8.80e-11}  & 21856 / 328 \\
\phiv($30$),  {\tt1e-1} & {\tt1.10e-11}  & 23104 / 340 \\
\phip($30$),  {\tt5e-1} & {\tt7.53e-07}  &  9900 / 160 \\
\phip($30$),  {\tt1e-2} & {\tt1.49e-06}  & 13980 / 221 \\
\phiRT($30$), {\tt1e-2} & {\tt3.09e-04}  &  6488 / 104 \\
\phiRT($30$), {\tt1e-3} & {\tt4.53e-06}  & 10230 / 165 \\
\phiv($20$),  {\tt5e-1} & {\tt1.53e-10}   & 30756 / 342 \\
\phip($20$),  {\tt5e-1} & {\tt8.21e-06}   & 14520 / 177 \\
\phiRT($20$), {\tt1e-2} & {\tt4.50e-04}   &  7433 /  95 \\
\phiRT($20$), {\tt1e-3} & {\tt9.78e-06}   & 15012 / 183 \\
\hline
\multicolumn{3}{c}{$802\times 802$ grid, $\Pe=20$, $t=5\cdot10^{-5}$}\\
\phiv($30$),  {\tt5e-1} & {\tt3.22e-11} & 6304 / 463 \\
\phip($30$),  {\tt1e-1} & {\tt5.03e-05} & 2310 / 186 \\
\phiRT($30$), {\tt1e-2} & {\tt3.51e-05} & 2100 / 167 \\
%
\phiv($20$),  {\tt1e-1} & {\tt1.53e-11} & 8690 / 495 \\
\phip($20$),  {\tt1e-2} & {\tt6.01e-06} & 4200 / 250 \\
\phiRT($20$), {\tt1e-3} & {\tt3.27e-06} & 3431 / 201 \\
\hline
\end{tabular}}
\end{table}

\begin{figure}
\includegraphics[width=0.4\textwidth]{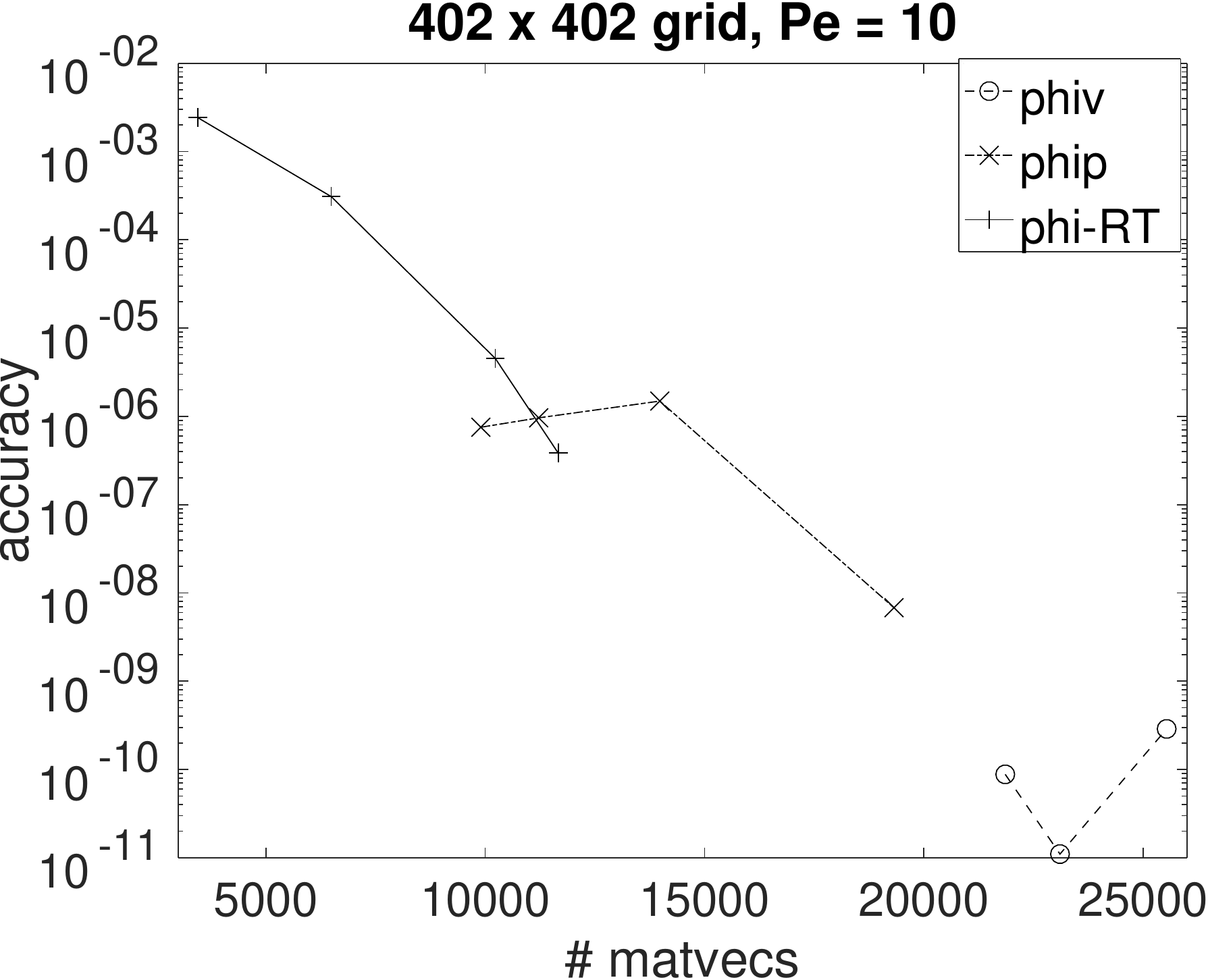}\hfill
\includegraphics[width=0.4\textwidth]{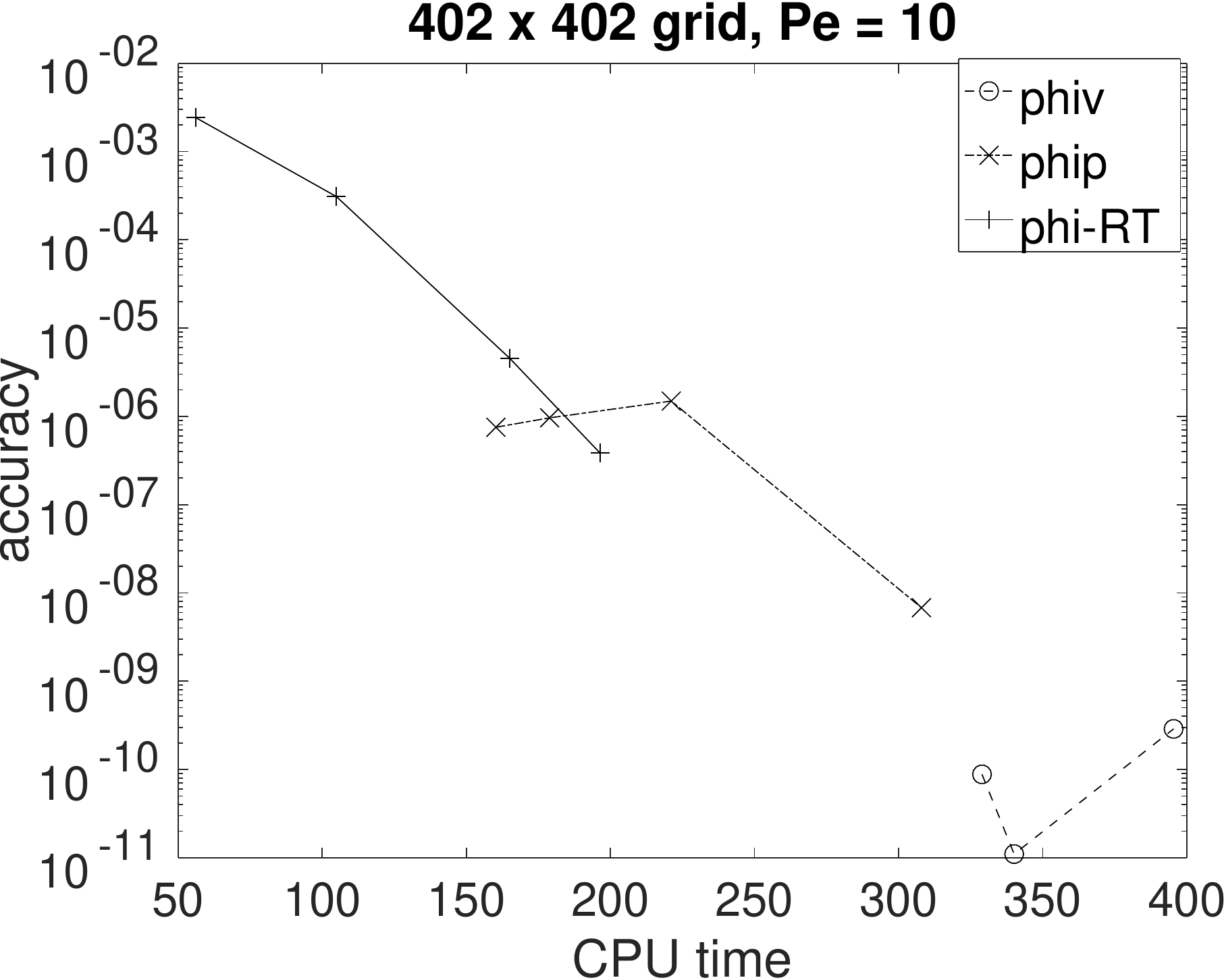}
\\[2ex]
\includegraphics[width=0.4\textwidth]{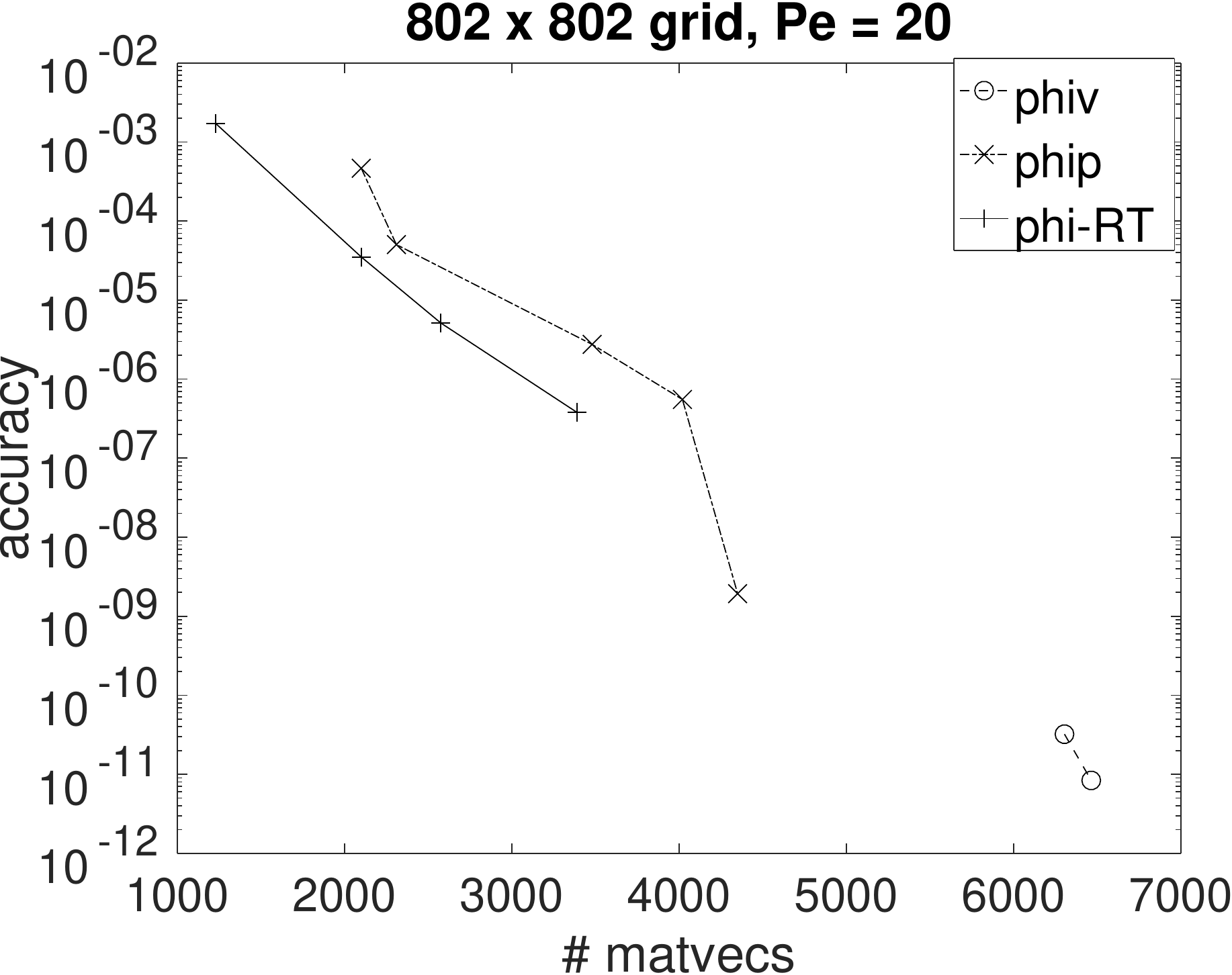}\hfill
\includegraphics[width=0.4\textwidth]{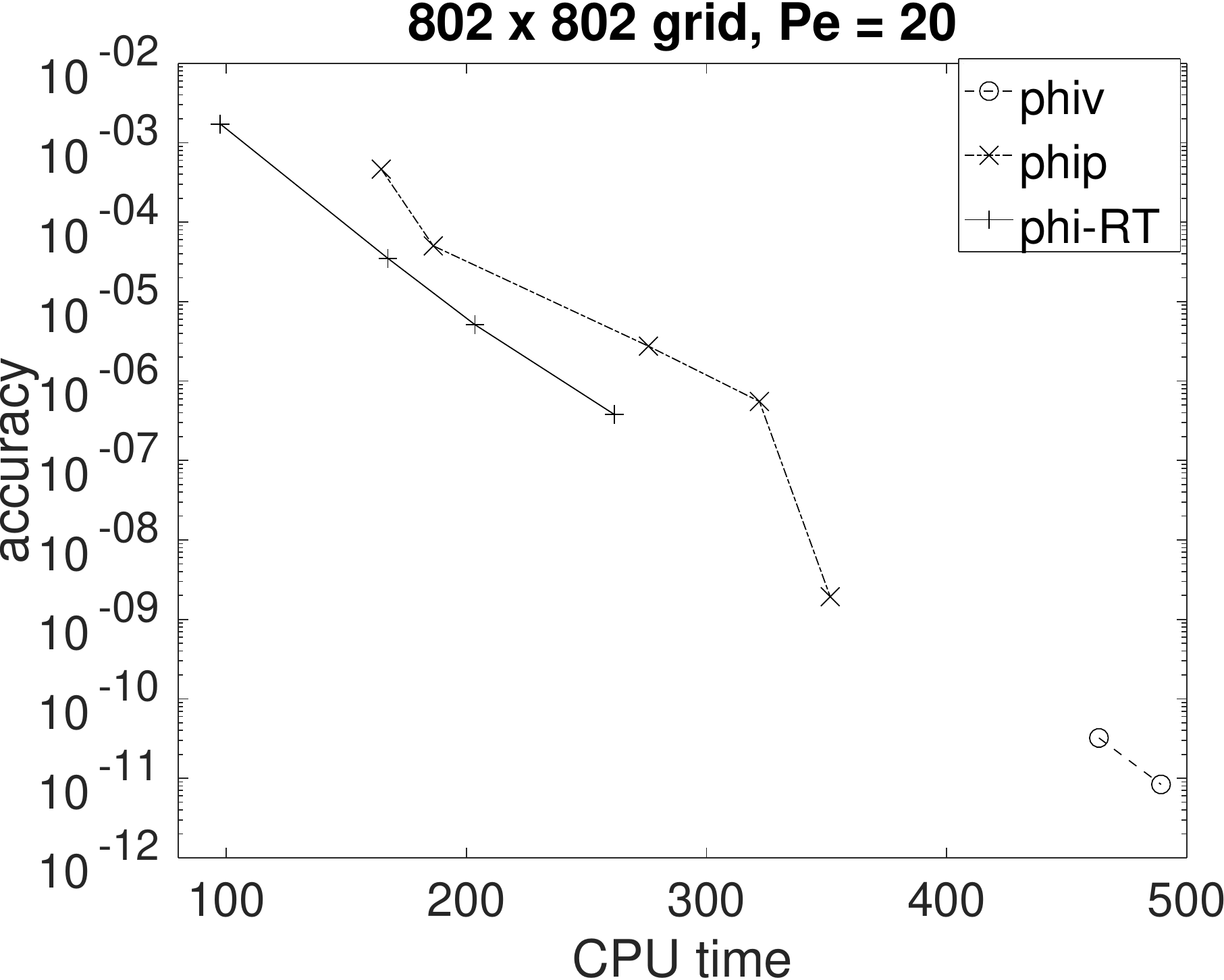}
\caption{Finite-difference discretized convection--diffusion
problem.
Relative error versus number
of the matrix-vector products (left)
and versus the CPU time (right) for \phiv{} (dashed line), \phip{} (dash-dotted line)
and \phiRT{} (solid line).
Top plots: $402\times 402$ grid, $\Pe=10$, $t=10^{-3}$,
bottom plots:
$802\times 802$ grid, $\Pe=20$, $t=5\cdot10^{-5}$.
The restart length is~30 for all the three solvers.
The loosest tolerance value for \phiv{} and \phip{} solvers
is {\tt5e-01}.
Hence, 
it is problematic to obtain cheaper, less accurate results
with these solvers.
}
\label{f:cd_acc}
\end{figure}

\subsection{Convection--diffusion discretized by finite elements,
  dominating convection}
\label{s:t3}
In this test IVP~\eqref{ivp} is solved where $A$ and $g$ are produced
by the IFISS software package~\cite{IFISS,IFISSacm2007,IFISS_sirev}.
More precisely, $A$ results from a finite-element discretization
of the following two-dimensional convection--diffusion problem:
\begin{equation}
\label{bp_ifiss}  
  -\nu\nabla^2u  + \vec{w}\cdot\nabla u = 0,\qquad
  u=u(x,y),\quad (x,y)\in[-1,1]\times[-1,1],
\end{equation}
where $\nu$ is viscosity parameter and Dirichlet boundary conditions
$u=0.5$ are imposed at all boundary points except for $y=1$, where
$$
u(x,1)=\frac12 + \frac12(1 -x^2).
$$
The initial value vector $v$ in~\eqref{ivp} is set to zero in this
test.

Discretization of~\eqref{bp_ifiss} in IFISS by
bilinear quadrilateral ($Q_1$) finite elements with 
the streamline upwind Petrov--Galerkin (SUPG) stabilization
leads to a linear system $Au=g$, where $g$ enforces
the Dirichlet boundary conditions.
Thus, the IVP we solve is a nonstationary version
of the stationary discretized convection--diffusion problem.
We use uniform $512\times 512$ and $1024\times1024$ grids
with viscosity values $\nu=1/6400$ and $\nu=1/12800$,
respectively.  In both cases the maximum finite element
grid Peclet number, estimated by IFISS as
$$
\frac{1}{2\nu}\min\left\{\frac{h_x}{\cos\alpha},
\frac{h_y}{\sin\alpha}\right\}\|\vec{w}\|_2,
\quad \alpha=\arctan\frac{w_y}{w_x},
$$
is $\approx 25$, where $h_{x,y}$ and $w_{x,y}$ are respectively
the element sizes and the components of the velocity $\vec{w}$. 
%

For these viscosity values we have
$\frac12\|A+A^T\|_1=1$,
$\frac12\|A-A^T\|_1\approx 0.0078$ on the $512\times 512$ grid
and
$\frac12\|A+A^T\|_1=1$,
$\frac12\|A-A^T\|_1\approx 0.0039$ on the $1024\times1024$ grid.
The length of the time interval is set to $t=2\cdot 10^4$.
For this $t$ the solution $y(t)$ of IVP~\eqref{ivp} still
differs significantly from the stationary solution $A^{-1}g$
to which $y(t)$ converges for growing $t$,
see Figure~\ref{ifiss:sol}.
The reference solution is computed by \phiv{} with tolerance
{\tt1e-13}.

\begin{figure}
\centering{%
\includegraphics[width=0.4\textwidth]{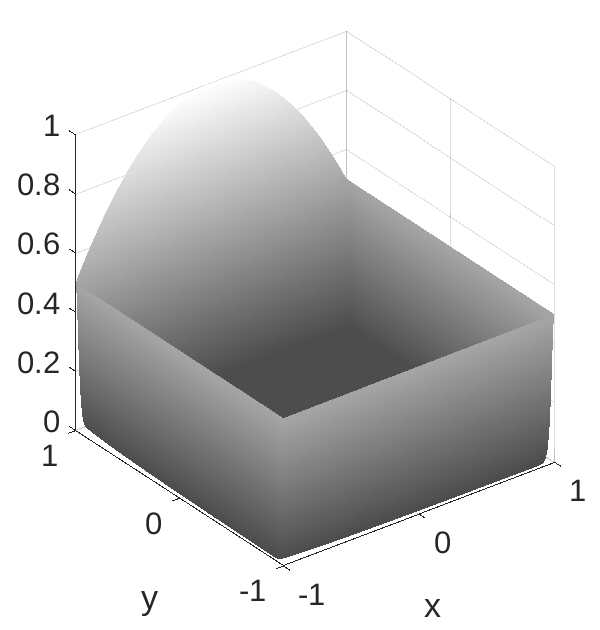}\hfill
\includegraphics[width=0.4\textwidth]{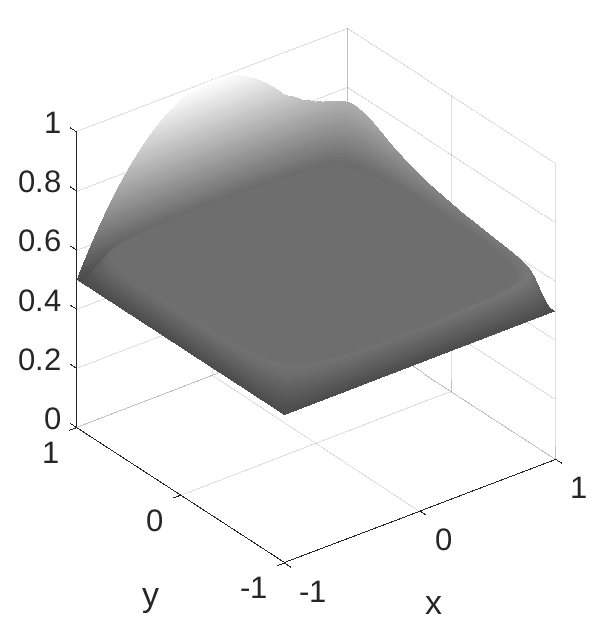}}      
\caption{Solution of the IFISS convection-dominated problem
$y(t)$ at $t=2\cdot 10^4$ (left) and the stationary
solution $A^{-1}g$ (right) on the $512\times 512$ grid.}
\label{ifiss:sol}  
\end{figure}

The results are presented in Tables~\ref{t:ifiss512}
and~\ref{t:ifiss1024} and in Figure~\ref{f:acc_ifiss}.
We see that \phiRT{} appears to be more efficient for this
test and that obtaining a time error within the desired tolerance
range $[10^{-7},10^{-3}]$ can be problematic with \phiv{} and \phip.
On both grids the \phip{} solver demonstrates a remarkable
``all-or-nothing'' behavior:  on the $1024\times 1024$ grid perturbing the
requested tolerance from {\tt7e-05} to {\tt8e-05} yields a drop
in the delivered accuracy from {\tt8.15e-07} to
{\tt5.36e-01}.  This last accuracy value is not plotted
in Figure~\ref{f:acc_ifiss} as it is too low
for a satisfactory PDE solution.

\begin{table}
\caption{Results for the finite-element discretized
convection-dominated convection--diffusion
test problem, $\nu=1/6400$, $512\times 512$ grid}
\label{t:ifiss512}
\centerline{\begin{tabular}{lcc}
\hline\hline
method(Krylov dim.),  & delivered & matvecs /   \\
requested tolerance   &   error   & CPU time, s \\
\hline
\phiv($30$),       {\tt1e-06} & {\tt1.36e-04}   & 160 / 5.6  \\
\phiv($30$),       {\tt1e-08} & {\tt1.25e-07}   & 224 / 7.7  \\
\phip($30$),       {\tt1e-04} & {\tt1.34e-01}   & 31  / 1.2 \\
\phip($30$),       {\tt1e-05} & {\tt1.40e-08}   & 210 / 7.7  \\
\phip($30$),       {\tt1e-06} & {\tt1.06e-09}   & 180 / 6.7  \\
\phiRT($30$),      {\tt1e-06} & {\tt2.60e-04}   & 89  / 3.3  \\
\phiRT($30$),      {\tt1e-08} & {\tt9.63e-07}   & 160 / 5.7  \\
\hline                                        
\phiv($10$),       {\tt1e-06} & {\tt2.69e-06}   & 216 / 4.3 \\
\phip($10$),       {\tt1e-06} & {\tt5.65e-11}   & 320 / 7.7 \\
\phiRT($10$),      {\tt1e-06} & {\tt3.65e-04}   & 166 / 4.1 \\
\hline
\end{tabular}}
\end{table}

\begin{figure}
\begin{center}
\includegraphics[height=0.37\textwidth]{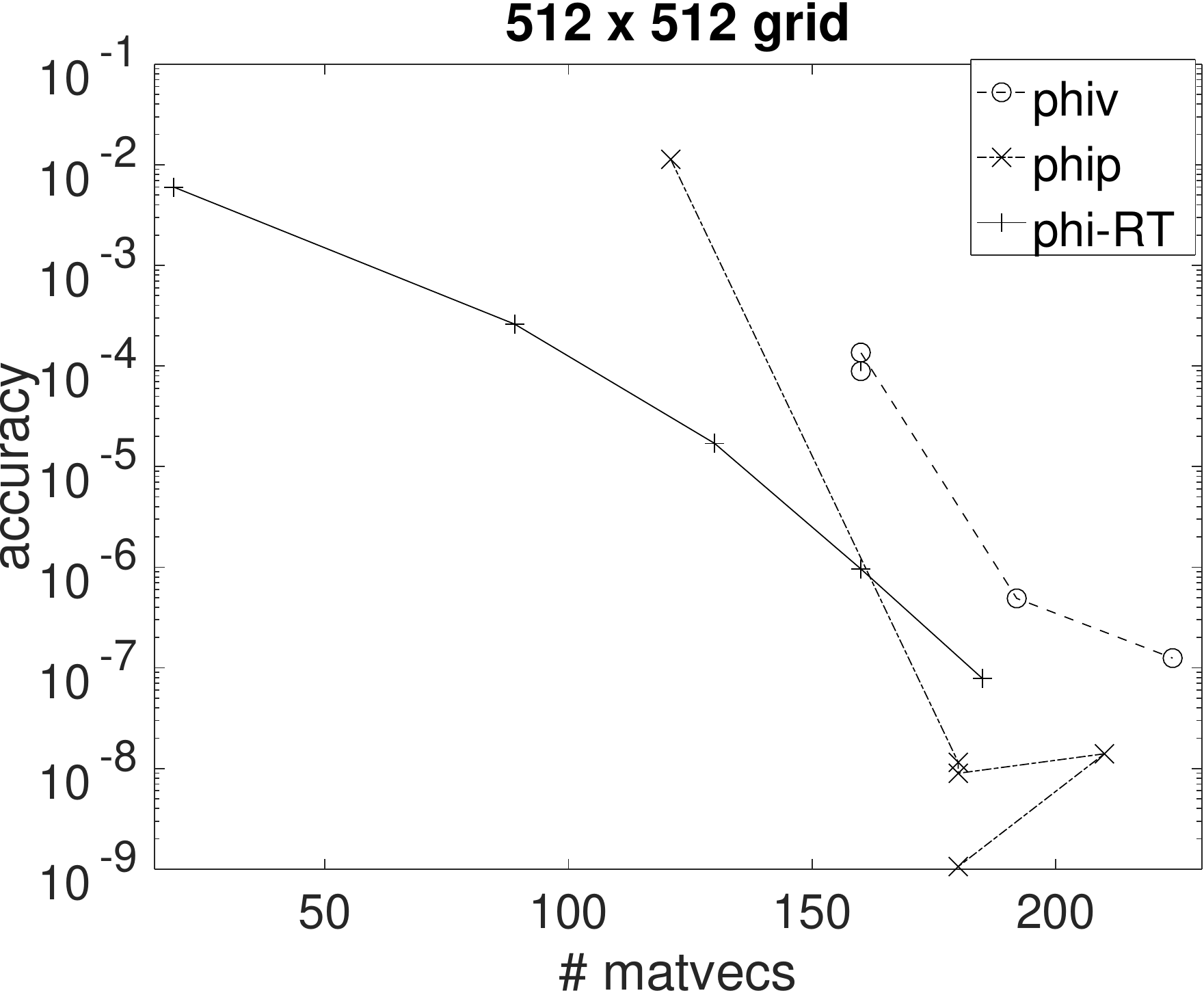}\hfill
\includegraphics[height=0.37\textwidth]{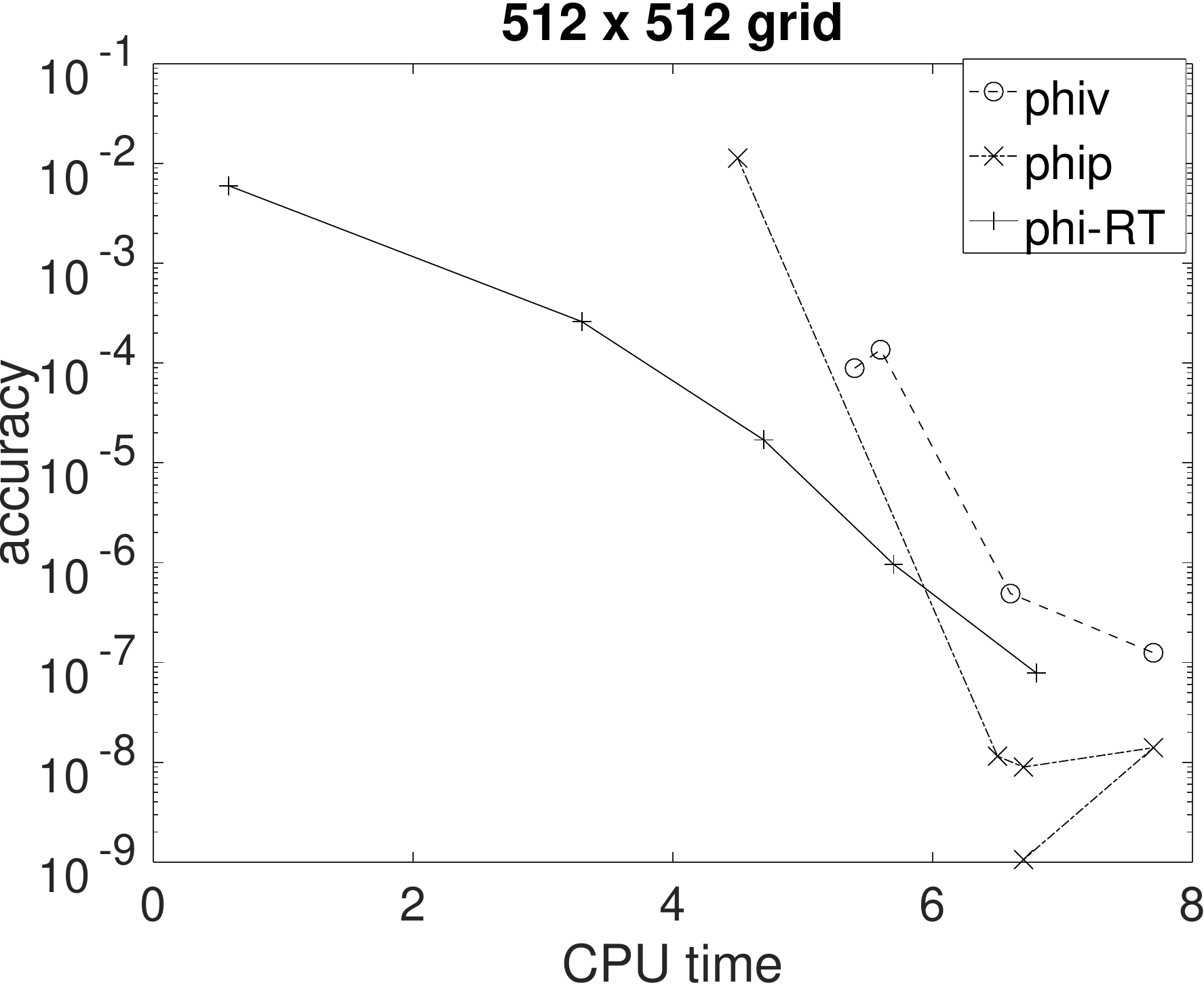}
\\[2ex]
\includegraphics[height=0.37\textwidth]{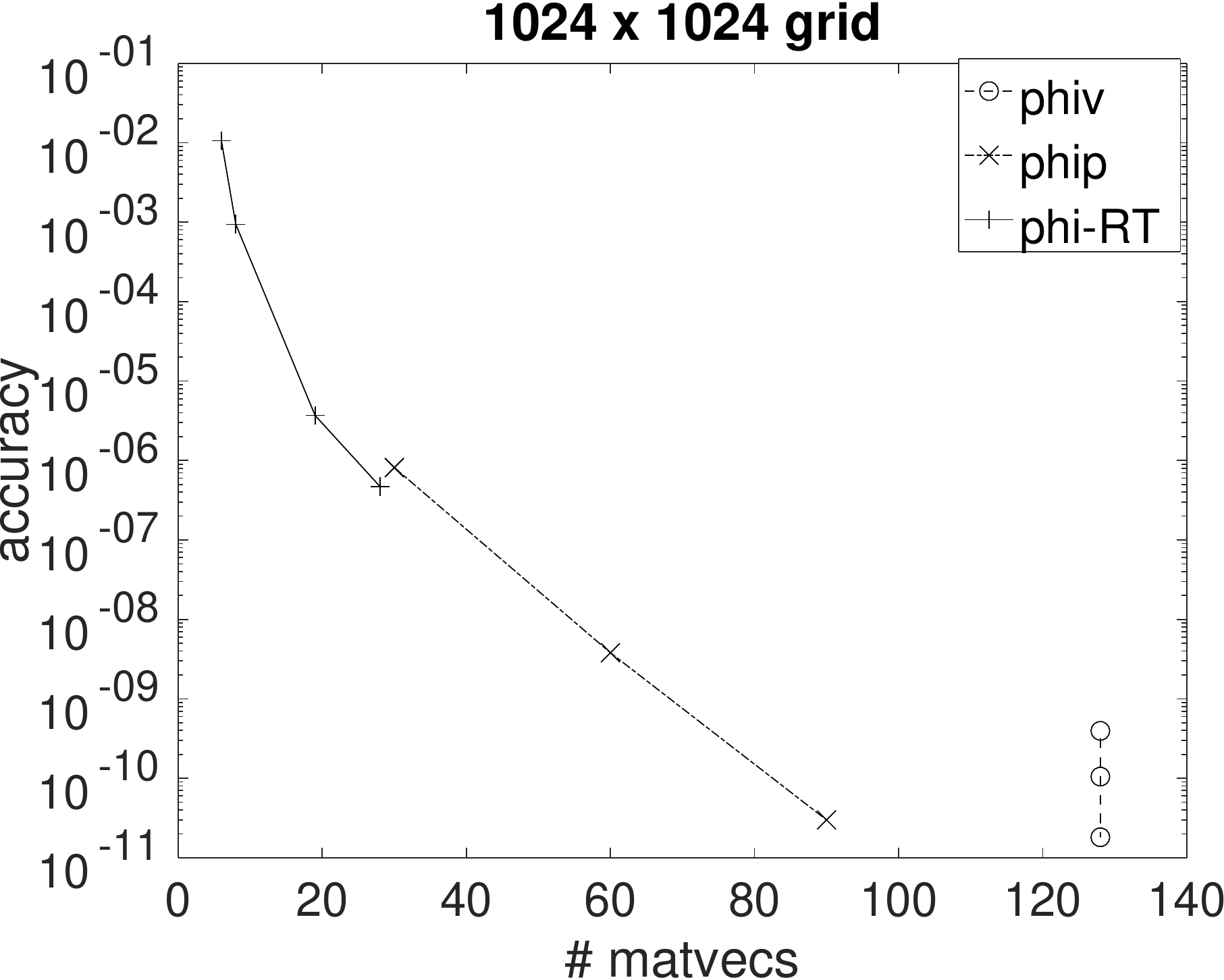}\hfill
\includegraphics[height=0.37\textwidth]{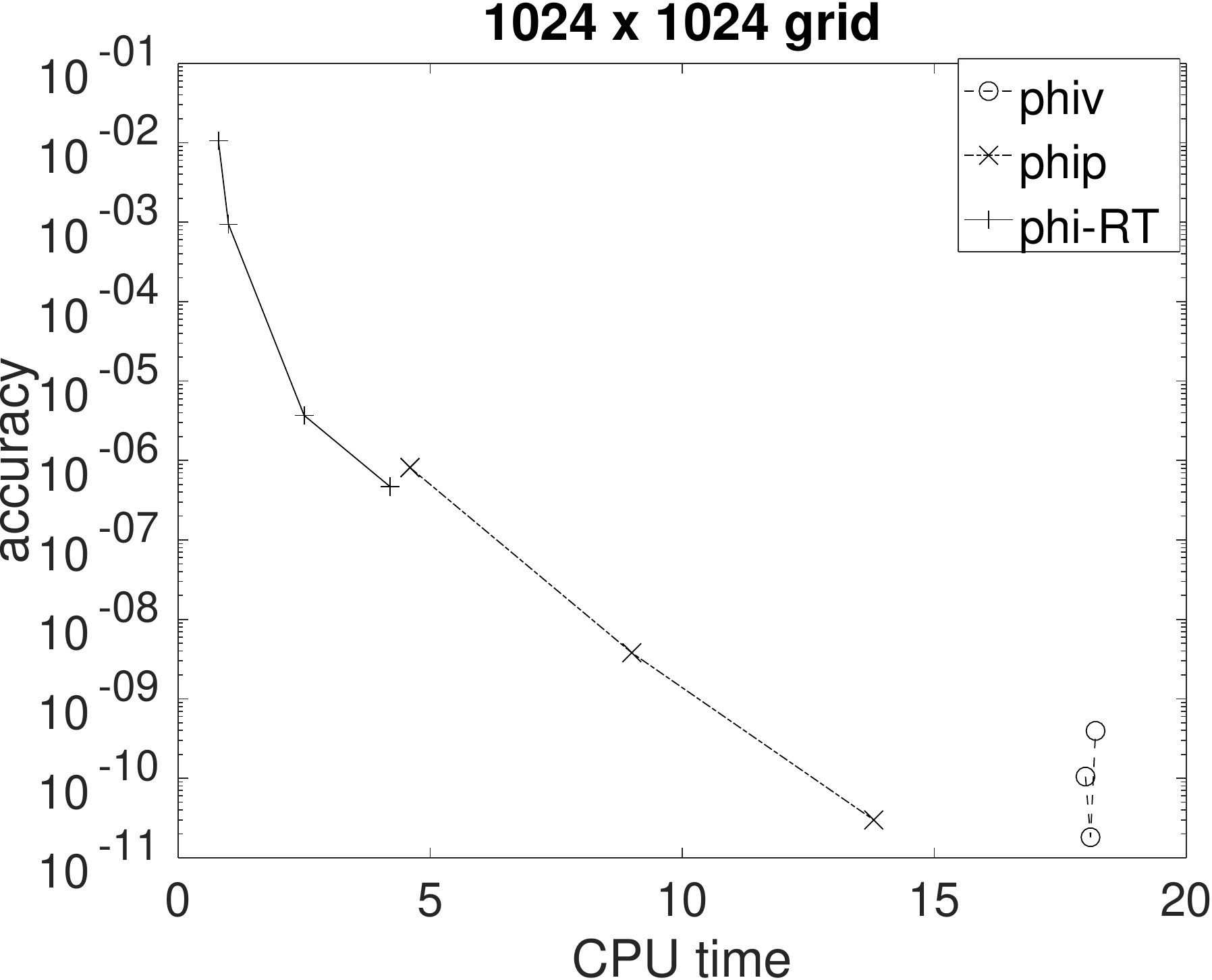}
\end{center}
\caption{The finite-element discretized
convection-dominated convection--diffusion test problem.
Relative error versus number of the matrix-vector products (left)
and versus the CPU time (right) for \phiv{} (dashed line), \phip{}
(dash-dotted line)
and \phiRT{} (solid line).
The restart length is~30 for all the three solvers.
Perturbing requested tolerance from {\tt6e-05} to {\tt7e-05}
for \phip{} yields a change in delivered accuracy from
{\tt1.15e-08} to {\tt1.13e-02}, see the two upper left points of the
dash-dotted line in the upper plots, $512\times 512$ grid.
A similar effect is observed for $1024\times 1024$ grid.
}
\label{f:acc_ifiss}
\end{figure}

\begin{table}
\caption{Results for the finite-element discretized
convection-dominated convection--diffusion
test problem, $\nu=1/12800$, $1024\times 1024$ grid}
\label{t:ifiss1024}
\centerline{\begin{tabular}{lcc}
\hline\hline
method(Krylov dim.),  & delivered & matvecs /   \\
requested tolerance   &   error   & CPU time, s \\
\hline
\phiv($30$),      {\tt1e-04} & {\tt1.81e-11}  & 128 / 18.1    \\
\phiv($30$),      {\tt1e-06} & {\tt3.95e-10}  & 128 / 18.2    \\
\phip($30$),      {\tt1e-04} & {\tt5.36e-01}  &  1 / 0.4      \\
\phip($30$),      {\tt1e-05} & {\tt3.80e-09}  &  60 / 9.0     \\
\phip($30$),      {\tt1e-06} & {\tt2.99e-11}  &  90 / 13.8    \\
\phiRT($30$),     {\tt1e-05} & {\tt1.06e-02}  &  6  / 0.8     \\
\phiRT($30$),     {\tt1e-06} & {\tt9.39e-04}  &  8  / 1.0     \\
\phiRT($30$),     {\tt1e-08} & {\tt3.70e-06}  & 19 / 2.5      \\
\phiRT($30$),     {\tt1e-09} & {\tt4.72e-07}  & 28 / 4.2      \\
\hline                                       
\phiv($10$),      {\tt1e-05} & {\tt3.54e-04}  & 60 / 5.2      \\
\phiv($10$),      {\tt1e-06} & {\tt8.86e-06}  & 96 / 8.2      \\
\phip($10$),      {\tt1e-04} & {\tt5.36e-01}  &  1 / 0.3      \\
\phip($10$),      {\tt1e-05} & {\tt1.39e-08}  & 100 / 9.9     \\
\phiRT($10$),     {\tt1e-05} & {\tt1.06e-02}  &  6 / 0.7      \\
\phiRT($10$),     {\tt1e-06} & {\tt9.39e-04}  &  8 / 0.9      \\
\phiRT($10$),     {\tt1e-07} & {\tt3.96e-05}  & 45 / 4.4      \\
\hline
\end{tabular}}
\end{table}

\section{Conclusions}
\label{s:concl}
We have proposed a Krylov subspace method for computing
matrix-vector products with the $\varphi$ matrix function.
It is based on the residual notion~\eqref{res}, which is essentially used
in our algorithm both as a stopping criterion and for efficient
restarting.
Our analysis provides true upper bounds for the residual, with no
asymptotic estimates, and shows convergence of the restarted method
for an arbitrary final time interval $[0,t]$, $t<\infty$, and
for any restart length (i.e., Krylov subspace dimension). 

Since the residual can be seen as the backward error,
the residual based stopping criterion and residual-time (RT) restarting
provide a reliable error control.  This is confirmed by the performance of
the method: in all the test runs it
either outperforms the \phiv{} and \phip{} solvers or delivers
a comparable efficiency.
More importantly, in all the test runs our method appears to deliver a
monotone dependence of the obtained accuracy on the required tolerance.
The \phiv{} and \phip{} solvers have in some cases difficulties with delivering
error in the desired moderate PDE tolerance
range $[10^{-7},10^{-3}]$; they tend to produce more accurate and more
expensive results than requested.

Thus, the presented residual based Krylov subspace method for
evaluating matrix-vector products with the $\varphi$ matrix function
appears to be a promising approach useful in practical computations.

\bibliography{matfun,mxw,my_bib}
\bibliographystyle{siamplain}
  
\end{document}

%% file: L_newcommands.tex
\def\be#1\ee{\begin{equation}#1\end{equation}}
\newcommand{\bea}{\begin{eqnarray}}
\newcommand{\eea}{\end{eqnarray}}
\newcommand{\beas}{\begin{eqnarray*}}
\newcommand{\eeas}{\end{eqnarray*}}
\newcommand{\Int}{\int\limits}

\newtheorem{remark}{Remark}

\newcommand{\tT}{{\mbox{\it \tiny T}}}

\renewcommand{\Re}{\mathop{\rm Re}\nolimits}

%% file: L_content.tex
\section{Estimates on the residual in terms of the Faber series}
\label{s:Faber}
For simplification of formulae we assume in this section that $\beta=1$.

Faber series as a means to investigate convergence of Arnoldi's method were introduced in \cite{Knizh91}; see also \cite{BR09}.

\subsection{A general estimate}
Denote by $W(A)$ be the numerical range of matrix $A$ and let
$\Phi_j$ be the Faber polynomials~\cite{Suetin} for the compact $W(A)$.
Since, for a fixed $t$, $\varphi(-tz)$ is an entire function of $z$,
we can consider the Faber series decomposition
\be \label{app3}
f(z;t) \equiv t\varphi(-tz) 
= \frac{1-e^{-tz}}{z} = \sum_{j=0}^{+\infty}f_j(t)\Phi_j(z), 
\qquad z\in W(A), \quad t\geqs0,
\ee
where $t$ is considered as a parameter. 

The following assertion is an analogue of \cite[Proposition~1]{BoKn2020}. 
\begin{proposition} \label{P1}
The residual (\ref{res}) satisfies the inequality
\be \label{appr2}
\|r_k(t)\|\leqs 2h_{k+1,k}\sum_{j=k-1}^{+\infty}|f_j(t)|
\leqs 2\|A\|\sum_{j=k-1}^{+\infty}|f_j(t)|.
\ee
\end{proposition}

\begin{proof}
The superexponential convergence in $j$ and the smoothness of $f_j(t)$ in $t$ enable one to differentiate series (\ref{app3}) in $t$.
Decomposition (\ref{app3}) implies
that the approximant (\ref{yk})
can be represented 
as
\be \label{app6}
y_k(t)=V_kf(H_k;t) e_1=V_k\sum_{j=0}^{+\infty} f_j(t)\Phi_j(H_k) e_1.
\ee
From~\eqref{app3} we see that
\begin{equation}
\label{etz}
e^{-tz}=1-zf(z;t)=1-\sum_{j=0}^{+\infty}f_j(t)z\Phi_j(z).  
\end{equation}
Therefore, since
$\frac{\partial e^{-tz}}{\partial t} = -ze^{-tz}$,
we obtain, differentiating~\eqref{etz},
\bea \label{app1}
-z e^{-tz} = \frac{\partial}{\partial t}\left[1-\sum_{j=0}^{+\infty}f_j(t)z\Phi_j(z)
\right],\nonumber\\
-z\left[1-\sum_{j=0}^{+\infty}f_j(t)z\Phi_j(z)\right]
=-\sum_{j=0}^{+\infty}f_j^\prime(t)z\Phi_j(z), \nonumber\\
1-\sum_{j=0}^{+\infty}f_j(t)z\Phi_j(z)
=\sum_{j=0}^{+\infty}f_j^\prime(t)\Phi_j(z), \nonumber\\
-1+\sum_{j=0}^{+\infty}\left[f_j^\prime(t)+zf_j(t)\right]\Phi_j(z) = 0.
\eea
Exploiting equality $\deg\Phi_j=j$ and relation (\ref{res}), (\ref{app6}),
(\ref{Arn}), (\ref{app1}) with $H_k$ substituted for $z$, we get
\beas
r_k(t) = -\left(V_kH_k+h_{k+1,k}v_{k+1}e_k^\tT\right)
\sum_{j=0}^{+\infty}f_j(t)\Phi_j(H_k) e_1 + g
-V_k\sum_{j=0}^{+\infty}f_j^\prime(t)\Phi_j(H_k) e_1 \\
=-V_k\sum_{j=0}^{+\infty}\left[f_j^\prime(t)\mathcal{I}_k+H_kf_j(t)\right]
\Phi_j(H_k) e_1 + V_k e_1
-h_{k+1,k}v_{k+1}e_k^\tT\sum_{j=0}^{+\infty}f_j(t)\Phi_j(H_k) e_1 \\
=-h_{k+1,k}v_{k+1}e_k^\tT\sum_{j=k-1}^{+\infty}f_j(t)\Phi_j(H_k) e_1,
\eeas
where $\mathcal{I}_k$ denotes the $k\times k$ identity matrix.
The bound $\|\Phi_j(H_k)\|\leqs2$ (see \cite[Theorem~1]{B05})
now implies (\ref{appr2}).
\end{proof}

\begin{remark}
Comparing estimate (\ref{appr2}) and the estimate on the error
\[
\|y(t)-y_k(t)\| \leqs 4\sum_{j=k}^{+\infty}|f_j(t)|
\]
(see \cite[Theorem~3.2]{BR09})
shows that the two upper bounds mainly differ in coefficients. Thus, we can conjecture that the error and the residual behave similarly to each other.
\end{remark}

\subsection{The symmetric case}

We assume here that $A^T=A$ and the spectral interval of $A$ is contained in the segment $[0,2c]$, $c>0$. The change of variable $x=(z-c)/c$ translates $f(z;t)$ into the function
\[
g(x;t) \equiv \frac{1-e^{-ct}e^{-ctx}}{c(x+1)},
\]
and we wish to estimate its Chebyshev coefficients in $x$, because in this case the Faber series essentially reduces to the Chebyshev series up to a linear variable scaling.
Following \cite{Pash}, we denote by $a_m[h]$ the $m$th Chebyshev coefficient of a function $h$ defined for $x\in[-1,1]$ as
\[
h(x) = \sum_{m=0}^{+\infty}{\!}' a_m[h]T_m(x),
\]
where $T_m$ is the $m$th first kind Chebyshev polynomial and the prime ${}\prime$ means that the term at $m=0$ is divided by two.

\begin{lemma} \label{L1}
For $k\geqs2$ the inequality
\be \label{resL1}
\sum_{m=k-1}^{+\infty} \left| a_m\left[
\frac{1-e^{-ct}e^{-ctx}}{c(x+1)} \right] \right|
\leqs \frac{2e^{-ct}}{c} \sum_{\nu=k}^{+\infty}
(\nu-k+2)(\nu-k+1)I_\nu(ct).
\ee
holds with Bessel functions $I_\nu$.
\end{lemma}

\begin{proof}
With formulae \cite[Ch.~II, \S~11, (10), (54)]{Pash}, we have
\[
e^{-ctx} = 2\sum_{j=0}^{+\infty}{\!}' I_j(-ct)T_j(x)
= 2\sum_{j=0}^{+\infty}{\!}' (-1)^jI_j(ct)T_j(x).
\]
Then, with $\delta_{i,j}$ denoting the Kronecker delta, we obtain
\[
a_m\left[\frac{1-e^{-ct}e^{-ctx}}{c}\right]
= \frac{2\delta_{m,0}-e^{-ct}a_m[e^{-ctx}]}{c}
= \frac{2\delta_{m,0}-2(-1)^me^{-ct}I_m(ct)}{c},
\]
which in view of \cite[Ch.~II, \S~9, Theorem~9.4, (42)]{Pash} gives
\beas
a_m\left[\frac{1-e^{-ct}e^{-ctx}}{c(x+1)}\right]
=2 \sum_{l=0}^{+\infty} (-1)^l {l+1 \choose 1} a_{m+l+1}
\left[\frac{1-e^{-ct}e^{-ctx}}{c}\right] \\
=\frac{4}{c} \sum_{l=0}^{+\infty} (-1)^l(l+1)
[\delta_{m+l+1,0}-e^{-ct}(-1)^{m+l+1}I_{m+l+1}(ct)] \\
=\frac{4}{c} (-1)^m \sum_{l=0}^{+\infty} (l+1)
e^{-ct}I_{m+l+1}(ct),
\eeas
whence
\beas
\sum_{m=k-1}^{+\infty} \left|a_m\left[ 
\frac{1-e^{-ct}e^{-ctx}}{c(x+1)} \right]\right|
\leqs \frac{4e^{-ct}}{c} \sum_{m=k-1}^{+\infty} \sum_{l=0}^{+\infty} (l+1) I_{m+l+1}(ct) \\
= \frac{4e^{-ct}}{c} \sum_{\nu=k}^{+\infty}
\left[\sum_{l=0}^{\nu-k}(l+1)\right] I_\nu(ct) 
= \frac{4e^{-ct}}{c} \sum_{\nu=k}^{+\infty}
\frac{(\nu-k+2)(\nu-k+1)}{2}I_\nu(ct) \\
= \frac{2e^{-ct}}{c} \sum_{\nu=k}^{+\infty}
(\nu-k+2)(\nu-k+1)I_\nu(ct).
\eeas
\end{proof}

\begin{proposition}[A general error bound for the symmetric case] \label{P2}
For $k\geqs2$
\be \label{resP2}
\|r_k(t)\| \leqs 4e^{-ct} \sum_{\nu=k}^{+\infty}
(\nu-k+2)(\nu-k+1)I_\nu(ct).
\ee
\end{proposition}

\begin{proof}
Gathering (\ref{appr2}), the definition of $c$, (\ref{resL1}) and the fact that $\Phi_{[-1,1],\nu}=2T_\nu$, $\nu\geqs1$ 
(see \cite[Ch.~II, \S~1, (21)]{Suetin}), derive
\beas
\|r_k(t)\| \leqs 4c \sum_{j=k-1}^{+\infty} |f_j(t)|
=2c \sum_{j=k-1}^{+\infty} \left|a_j\left[ 
\frac{1-e^{-ct}e^{-ctx}}{c(x+1)} \right]\right| \\
\leqs 4e^{-ct} \sum_{\nu=k}^{+\infty}
(\nu-k+2)(\nu-k+1)I_\nu(ct).
\eeas
\end{proof}

A similar use of Bessel functions $I_k$ can be found in \cite[\S~4]{DruskinKnizh89}. Here we have the term $(\nu-k+2)(\nu-k+1)$ as a complication.
Note that the upper bound~\eqref{resP2} can be easily computed with a high
accuracy by truncating the infinite sum $\sum_{\nu=k}^{+\infty}$
to $\sum_{\nu=k}^{M}$ where $M$ can be estimated as discussed
in~\cite[\S~4]{DruskinKnizh89}.  Nevertheless, as Proposition~\ref{P3}
below shows, an explicit bound on the residual norm can be obtained
from~\eqref{resP2}.

\begin{lemma}
Let $\mu\in\mathbb{N}$ and $w\geqs0$. Then the Bessel function $I_\mu$ obeys the inequality
\be \label{resL2}
e^{-w}I_\mu(w) \leqs \frac{w^\mu}{(2\mu-1)!!}.
\ee
\end{lemma}

\begin{proof}
Using formulae \cite[9.6.18, 6.1.8, 6.1.12]{AbramowitzStegun}, derive
\beas
e^{-w}I_\mu(w) \leqs \frac{(w/2)^\mu}{\sqrt\pi\,\Gamma(\mu+1/2)}
\left| \Int_0^\pi e^{w\cos\theta}\sin^{2\mu}\theta\,d\theta \right| 
\leqs e^{-w}\frac{(w/2)^\mu}{\sqrt\pi\,\Gamma(\mu+1/2)} \,\pi e^w \\
=\frac{\sqrt\pi w^\mu}{2^\mu\cdot\frac{(2\mu-1)!!}{2^\mu}\cdot\sqrt\pi} 
=\frac{w^\mu}{(2\mu-1)!!}.
\eeas
\end{proof}

\begin{proposition}[An error bound for $k$ of order at least $ct$
in the symmetric case]
\label{P3}
If
\be \label{assP2}
k\geqs\max\{5ct,2\},
\ee
then
\be \label{resP3}
\|r_k(t)\| \leqs \frac{16(ct)^k}{(2k-1)!!} 
\ee
\end{proposition}

\begin{proof}
Exploiting assumption (\ref{assP2}) and the inequality 
(see \cite[Proposition~1]{Nasell74})
\[
I_{\nu+1}(w)<\frac{I_\nu(w)}{1+\frac{\nu}{w}},
\qquad w>0, \qquad \nu\geqs-1,
\]
obtain for $\nu\geqs k$
\beas
\frac{(\nu-k+3)(\nu-k+2)I_{\nu+1}(ct)}{(\nu-k+2)(\nu-k+1)I_{\nu}(ct)}
= \left(1+\frac{2}{\nu-k+1}\right)
\cdot\frac{I_{\nu+1}(ct)}{I_\nu(ct)} \\
<\frac{3}{1+\frac{\nu}{ct}} \leqs \frac{3}{1+\frac{5ct}{ct}}
\leqs \frac{1}{2},
\eeas
so that, in view of (\ref{resL2}),
\[
e^{-ct} \sum_{\nu=k}^{+\infty} (\nu-k+2)(\nu-k+1)I_\nu(ct)
\leqs e^{-ct} \cdot 4I_k(ct)
\leqs \frac{4(ct)^k}{(2k-1)!!}.
\]
It remains to substitute the above bound into (\ref{resP2}).
\end{proof}

Proposition~\ref{P3} is an analogue of \cite[Theorem~5]{DruskinKnizh89}.
Bound~(\ref{resP3}) would be exponentially worse, if we used the Taylor series instead of the Chebyshev one; this clarifies the role of Faber series.

In Figure~\ref{estsymm} we plot the graph of the right-hand side
of~(\ref{resP2})
for the values $c$ and $t$ as occur in the numerical test presented
in Section~\ref{s:t1}.  Since $\|A\|_2^2\leqs\|A\|_1\|A\|_{\infty}$
(see, e.g., \cite[Chapter~5.6]{HornJohnsonI}) and $\|A\|_1=\|A\|_{\infty}$
for $A=A^T$,
we can estimate $c$ by $\|A\|_1/2$. This leads to the $c$ values
reported in caption of Figure~\ref{estsymm}.  The $t$ values used
in the plots (and reported in the caption) are taken to be approximately the average
time interval length $\bar{\delta}$ per restart, i.e., $t:=\bar{\delta}$, obtained from data in Table~\ref{t:p4}
as the total time interval length divided by the number of restarts.
In the figure we also plot the actual residual norm
values computed with the regular Lanczos process.

\begin{figure}[h]
\centerline{%
\includegraphics[width=0.48\textwidth]{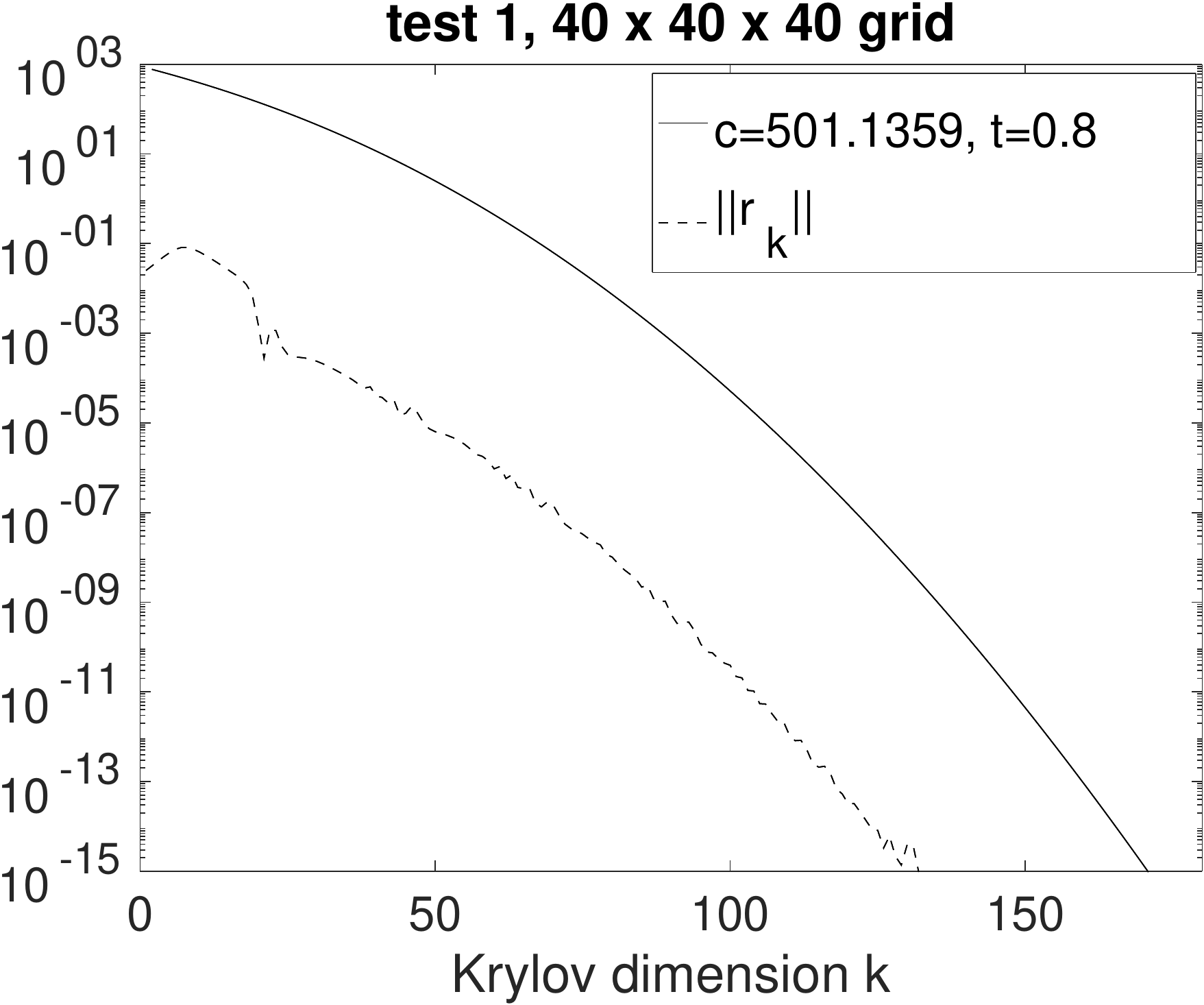}
\includegraphics[width=0.48\textwidth]{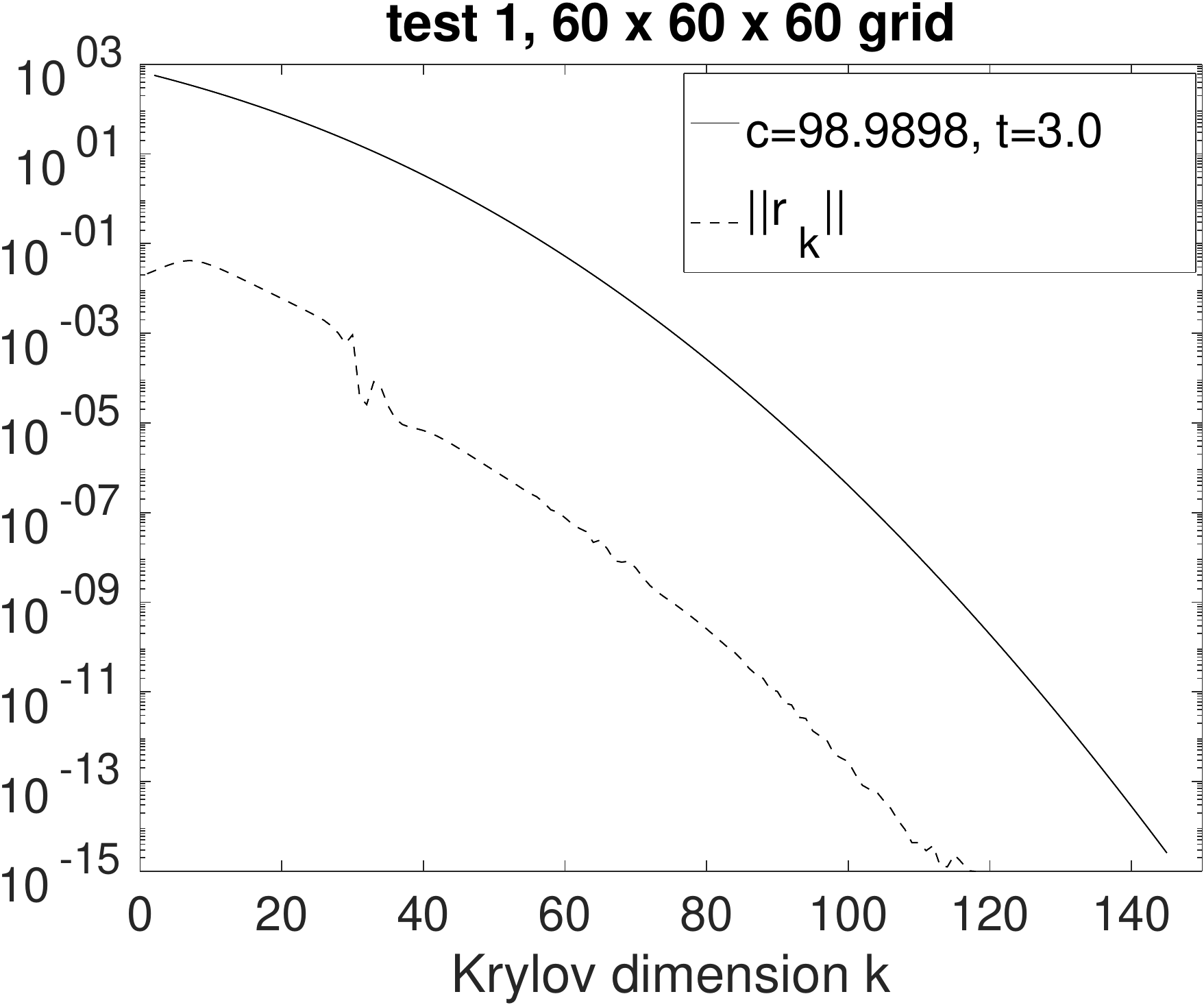}}
\caption{The solid lines are error upper bound~\eqref{resP2} for
  $c=501.1359$, $t=\bar{\delta}=0.8$ (left plot) and
  $c=98.9898$, $t=\bar{\delta}=3.0$ (right plot).
The values correspond to $A$ being discretized anisotropic elliptic operator described
in Section~\ref{s:t1}, for $40\times 40\times 40$ (left plot)
and $60\times 60\times 60$ (right plot) finite volume grids, respectively.
The dashed lines are the actual scaled residual norm values $\|r_k\|/\beta$.}
\label{estsymm}
\end{figure}

From Figure~\ref{estsymm} we see that the obtained upper bound is far from
being sharp.  This is certainly expected.  Although~\eqref{resP2} is a true
upper bound and contains no asymptotic estimates, it does not reflect
the sensibility of the Krylov subspace methods with respect to the
discrete structure of the spectrum.
To confirm this and to show that the upper bound in~\eqref{resP2} can be
sharp, in Figure~\ref{f:worst} we plot the upper bound
along with the actual residual norm values for the following
tridiagonal $A\in\Rrnn$ and $g\in\Rr^n$:
\begin{equation}
\label{worst}  
A = 
\begin{bmatrix}
  1          & 1/\sqrt{2} & 0      & \dots  & 0     \\
  1/\sqrt{2} &      1     & 1/2    & \ddots &\vdots \\
  0          &  1/2       & \ddots &\ddots  &       \\
  \vdots     &  \ddots    & \ddots & \ddots &
\end{bmatrix},
\quad g =
\begin{bmatrix}
1 \\ 0 \\ \vdots \\ 0  
\end{bmatrix}.
\end{equation}
It can be shown that $\|A\|\approx 2$ for large $n$, hence, we take
$n=1000$ and $c=1$.
The residual values are computed with the regular Lanczos process.
As we see from the plots, the upper bound~\eqref{resP2} is sharp.

\begin{figure}[h]
\centerline{%
\includegraphics[width=0.48\textwidth]{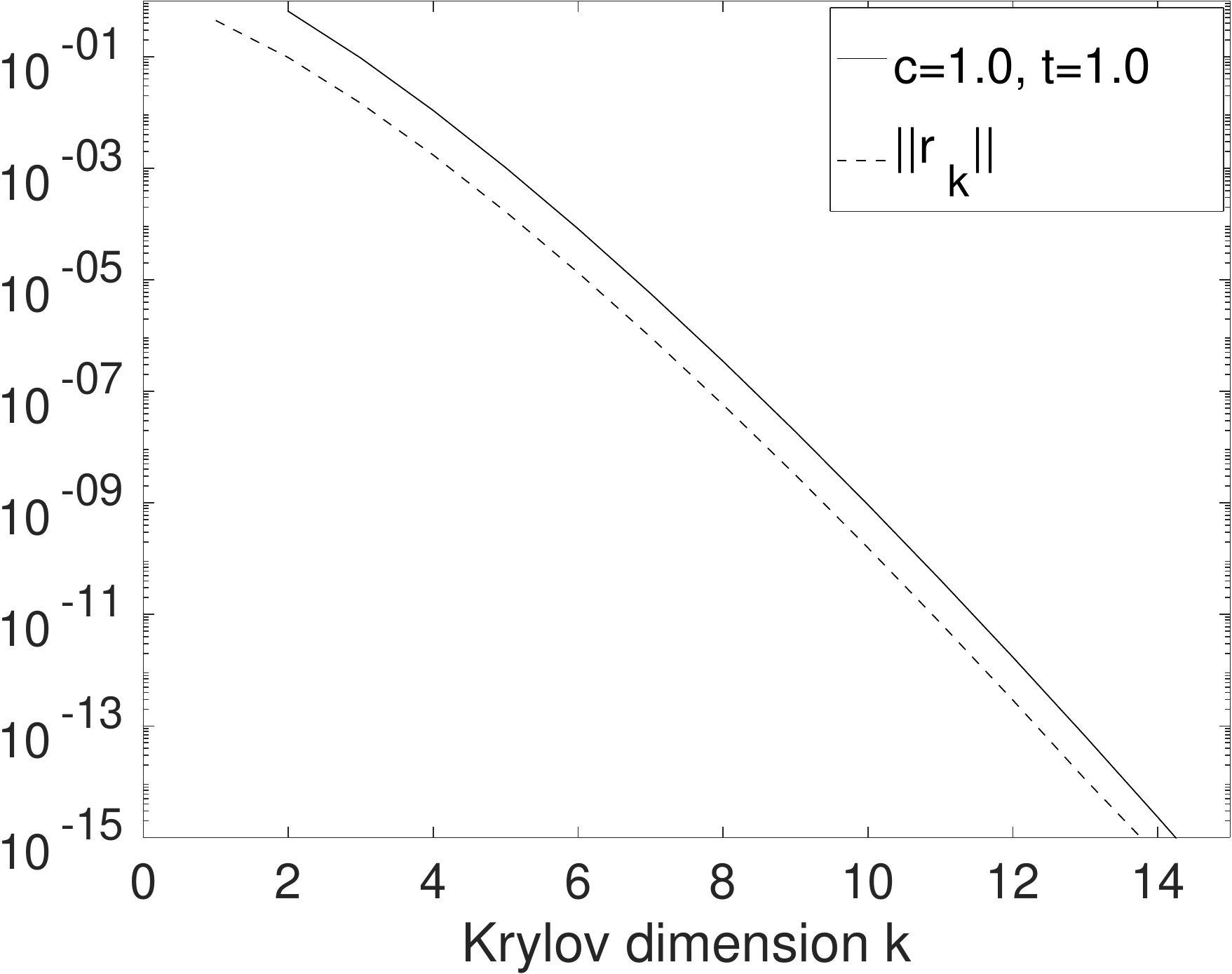}
\includegraphics[width=0.48\textwidth]{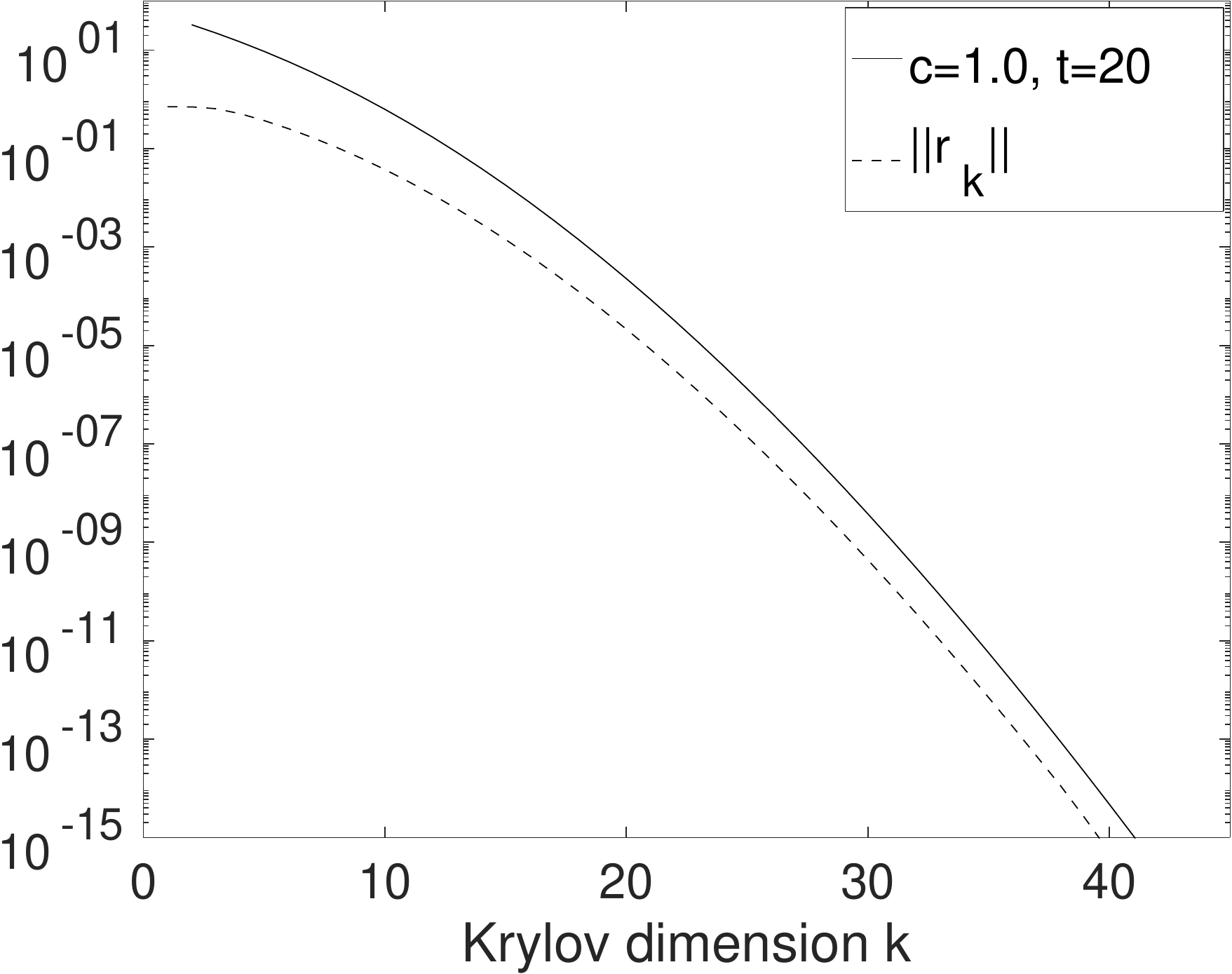}}
\caption{The solid lines are error upper bound~\eqref{resP2} for
  $c=1$, $t=1$ (left plot) and
  $c=1$, $t=20$ (right plot).
The values correspond to $A$ defined in~\eqref{worst}.
The dashed lines are the actual scaled residual norm values
$\|r_k\|/\beta$ obtained for $g$ defined in~\eqref{worst}.}
\label{f:worst}
\end{figure}

\subsection{A weakly non-symmetric case}
\begin{proposition}[An error bound for an ``elliptic'' numerical range] \label{P4}
Let the numerical range $W(A)$ be enlcosed by an
ellipse $E_{c,\rho}$ with foci $0$ and $2c$ with the sum of semiaxes $\rho c$, $c>0$, $\rho>1$. Then for $k\geqs2$
\be \label{resP4}
\|r_k(t)\| \leqs 4e^{-ct} \sum_{\nu=k}^{+\infty}
(\nu-k+2)(\nu-k+1)\rho^\nu I_\nu(ct).
\ee
\end{proposition}

\begin{proof}
The proof is a slight generalization of that of proposition~\ref{P2}; the latter corresponds to the limit case $\rho=1$. One should exploit the relation (see \cite[Ch.~II, \S~1, (25)]{Suetin}) between the Faber polynomials for $E_{c,\rho}$ and the Faber polynomials for the focal segment $[0,2c]$ (actually, the shifted Chebyshev polynomials).
\end{proof}

Similarly to~\eqref{resP2}, upper bound~\eqref{resP4} is
easily computable by truncating the infinite series in~\eqref{resP4}.
  
In Figure~\ref{estnonsymm} we plot the right-hand side of (\ref{resP4})
for the matrices of the second and third test problems,
see Sections~\ref{s:t2} and~\ref{s:t3}.
In each of these two test problems several grids and corresponding parameter settings
are used, and the two matrices for the plots in Figure~\ref{estnonsymm}
are built for the grids $202\times 202$
(the second test) and $512\times 512$ (the third test), respectively.
The $t$ values are approximate average time interval length $\bar{\delta}$ per restart,
computed as the total time interval length divided by the number of restarts.
To obtain the ellipse parameters appearing in the bound,
we approximate the numerical ranges of the matrices $A$ by the
numerical ranges of their Krylov projections $H_k=V_k^TAV_k$, see~\eqref{Arn}.
This is done for a sufficiently large $k$ such that the exterior eigenvalues of $A$
are well approximated by the eigenvalues of $H_k$.
Thus, the plots in Figure~\ref{estnonsymm} are produced
for approximate spectral data but do suffice for illustrative purposes.
In the figure the actual residual norm is also plotted.

\begin{figure}[h]
\centerline{\includegraphics[width=0.48\textwidth]{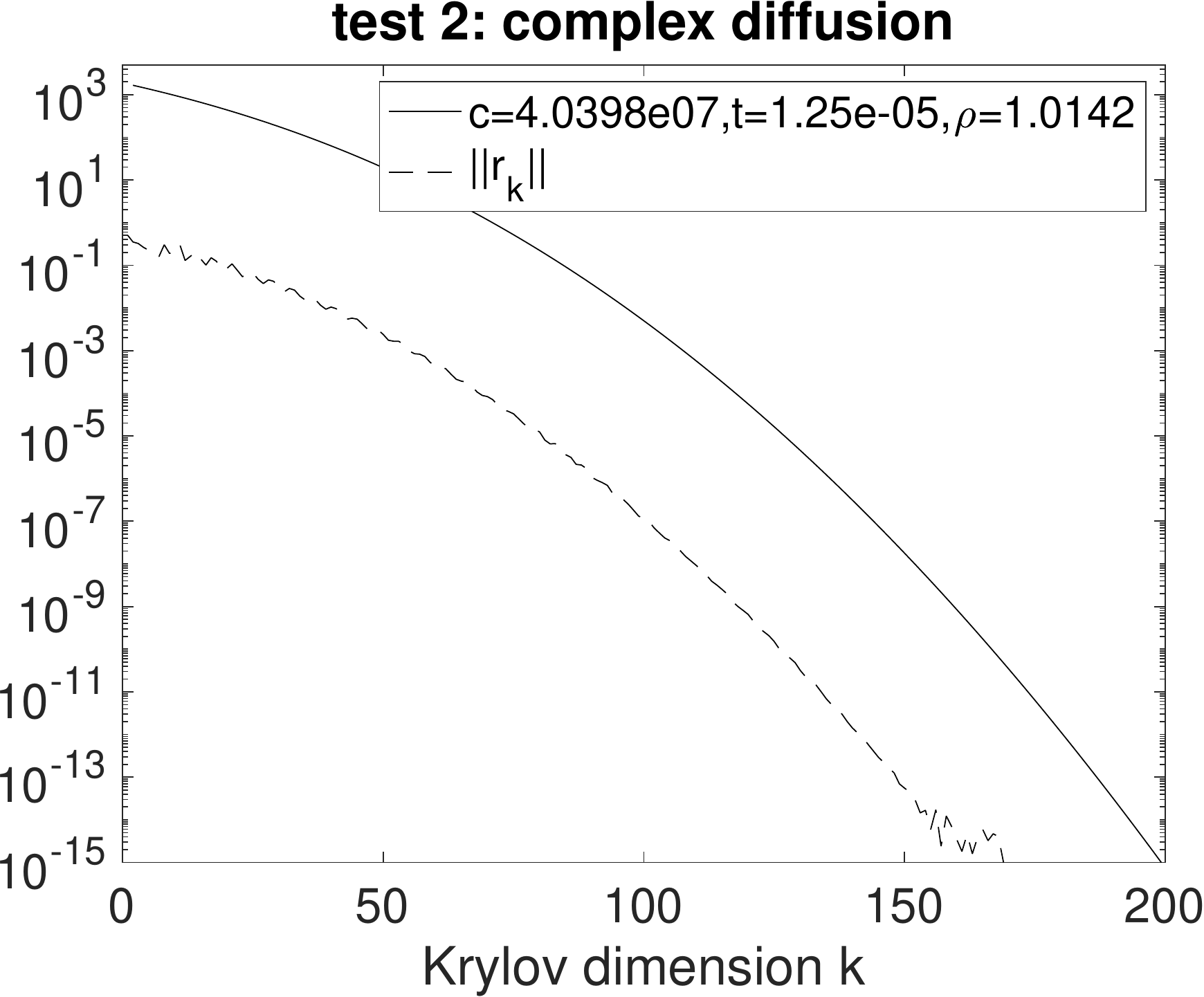}
\includegraphics[width=0.48\textwidth]{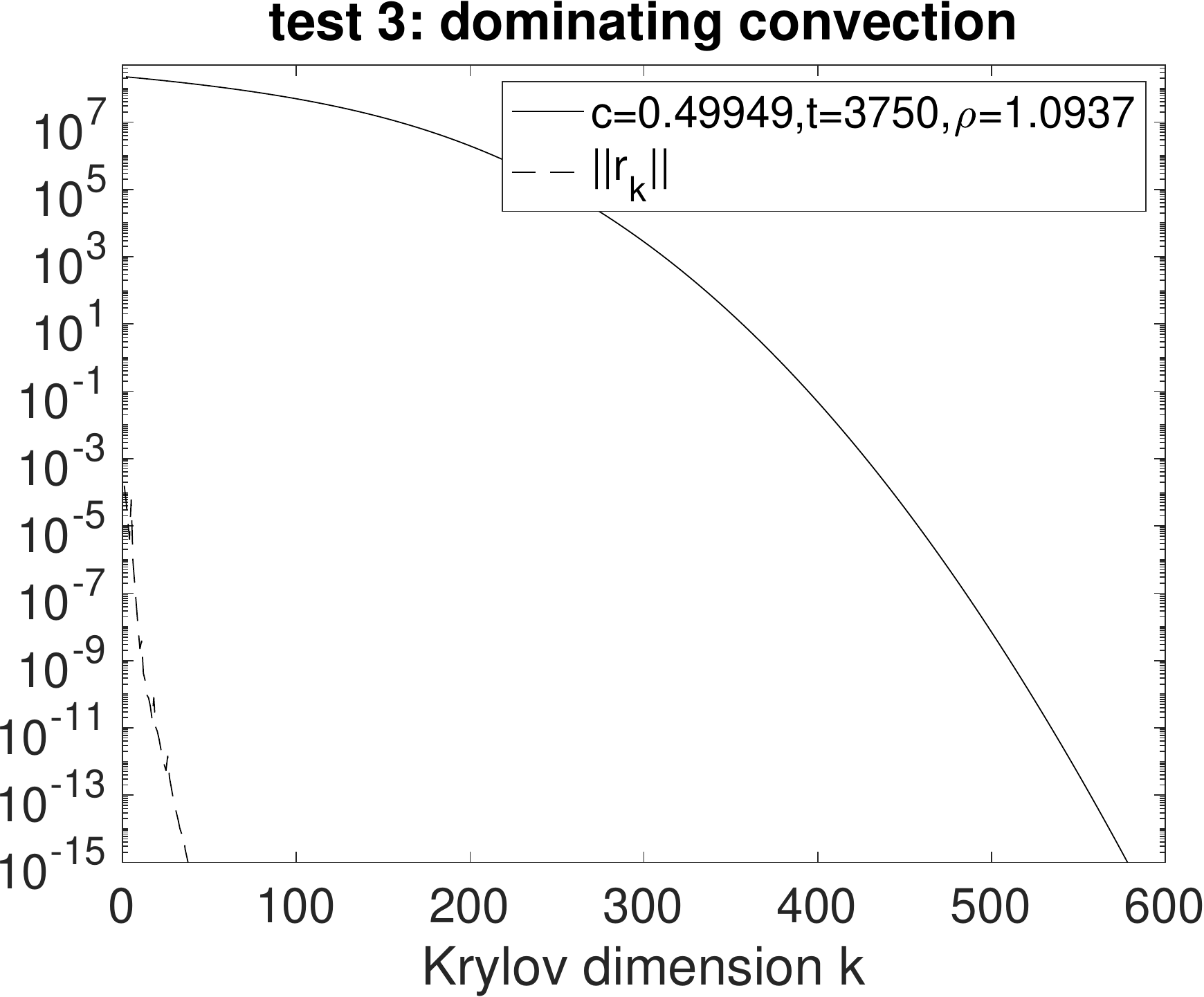}}
\caption{The solid lines are error upper bound~\eqref{resP4} for
$c={\tt4.0398e07}$, $t=\bar{\delta}=\texttt{1.25e-05}$, $\rho=1.0142$ (left plot) and
  $c=0.49949$, $t=\bar{\delta}=3750$, $\rho=1.0937$ (right plot).
  The values
  correspond to the matrices $A$ of the second
  test problem (see Section~\ref{s:t2}, grid $202\times 202$) and the third
  test problem (see Section~\ref{s:t3}, grid $512\times 512$), respectively.
  The dashed lines are the actual scaled residual norm values $\|r_k\|/\beta$.}
\label{estnonsymm}
\end{figure}